\documentclass{amsart}
\usepackage{a4}%
\usepackage{tikz,tikz-cd}
\usepackage{url}
\usepackage{graphicx}
\usepackage{color}
\usepackage{times}
\usepackage{cite}
\usepackage{enumerate,latexsym}
\usepackage{latexsym}
\usepackage{accents}
\usepackage{nicefrac}
\usepackage{amsmath,amssymb}
\usepackage{amsthm}
\usepackage{verbatim}

\newtheorem{theorem}{Theorem}[section]
\newtheorem*{theorem*}{Theorem}
\newtheorem*{acknowledgement*}{Acknowledgement}
\newtheorem{corollary}[theorem]{Corollary}

\newtheorem{lemma}[theorem]{Lemma}
\newtheorem{proposition}[theorem]{Proposition}
\theoremstyle{definition}
\newtheorem{definition}[theorem]{Definition}
\theoremstyle{remark}
\newtheorem{remark}[theorem]{Remark}

\numberwithin{equation}{section}

\renewcommand{\epsilon}{\varepsilon}
\newcommand{\circo}{\accentset{\circ}}
\DeclareMathOperator{\tr}{tr}

\DeclareMathOperator{\supp}{supp}

\author{Stephen Lynch}
\address{Eberhard Karls Universit\"{a}t T\"{u}bingen \\ Auf der Morgenstelle 10\\ 72076 T\"{u}bingen\\ Germany}
\email{stephen.lynch@math.uni-tuebingen.de}
\author{Huy The Nguyen}
\address{Queen Mary University of London\\ Mile End Road\\ London E1 4NS\\ United Kingdom}
\email{h.nguyen@qmul.ac.uk}

\begin{document} 

\title{Convexity Estimates for High Codimension Mean Curvature Flow}
\maketitle

\begin{abstract}
We consider the evolution by mean curvature of smooth $n$-dimensional submanifolds in $\mathbb{R}^{n+k}$ which are compact and quadratically pinched. We will be primarily interested in flows of high codimension, the case $k\geq 2$. We prove that our submanifold is asymptotically convex, that is the first eigenvalue of the second fundamental form in the principal mean curvature direction blows up at a strictly slower rate than the mean curvature vector. We use this convexity estimate to show that at a singular time of the flow, there exists a rescaling that converges to a smooth codimension-one limiting flow which is convex and moves by translation. 
\end{abstract}

\section{Introduction}

Let us consider a compact smooth $n$-manifold $M$, and a family of immersions
\[F : M\times [0,T) \to \mathbb{R}^{n+k}\]
which move by mean curvature flow, that is
\[\partial_t F(x,t) = H(x,t)\]
for each $(x,t) \in M\times [0,T)$ where $H$ is the mean curvature vector. The mean curvature flow constitutes a system of quasilinear weakly parabolic partial differential equations for $F$, and since $M$ is compact the flow must form a singularity in finite time. Singularity formation may be characterised analytically as follows: if we let $A$ denote the second fundamental form and take $T$ to be the maximal time then there holds
\[\limsup_{t\nearrow T} \sup_M |A|(\cdot,t) = \infty.\]
A profound and challenging problem is characterising and classifying the geometry of singularities, whose formation depends on the initial submanifold $M_0$ where $M_t:=F(M,t)$. In the seminal work of Huisken-Sinestrari \cite{Huisk-Sin99a}, the singularities formed by mean convex codimension one solutions were shown to be weakly convex (White obtained a similar result for embedded mean convex solutions in \cite{White}). It is natural to seek a corresponding theorem for solutions of higher codimension. However, we encounter a number of new difficulties, the foremost being that the second fundamental form and the mean curvature are vector-valued and consequently there is no direct corresponding notion of mean convexity. We instead use a different but related condition introduced by Andrews-Baker\cite{Andrews2010}. They showed that when $n \geq 2$, the quadratic pinching condition
\begin{equation}\label{eqn_quadratic} 
|A|^2 - c |H|^2 + a\leq 0
\end{equation}
is preserved for each $c < \frac{4}{3n}$ and $a > 0$. That is, if the condition is satisfied at the initial time, then it is satisfied by $M_t$ for every $t\in[0,T)$. We note that for compact hypersurfaces, positive mean curvature implies this condition but for all sufficiently large $c$. We will refer to submanifolds satisfying \eqref{eqn_quadratic} as being quadratically pinched. 

Andrews-Baker showed that if $c < \{\frac{4}{3n}, \frac{1}{n-1}\}$ the flow contracts quadratically pinched solutions to round points. This result is a high codimension generalisation of Huisken's work on convex solutions of codimension one \cite{Huisken84}. Recently the second-named author has constructed a flow with surgeries for solutions of dimension $n \geq 5$ which are quadratically pinched with 
\[c< c_n : =\begin{cases} 
      \frac{3(n+1)}{2n(n+2)} & n = 5, 6, 7, \\
      \frac{4}{3n} & n \geq 8.
   \end{cases}
\]
This generalises the surgery construction for two-convex hypersurface flows due to Huisken-Sinestrari \cite{Huisk-Sin09}. An important ingredient in \cite{Nguyen20} (and in the present work) is the codimension estimate due to Naff, which implies the singularities formed by a quadratically pinched solution are codimension one if $c < c_n$.

In this paper we show that a quadratically pinched mean curvature flow with $c < c_n$ is asymptotically convex in a quantifiable manner. As the second fundamental form is vector-valued, we denote by $\lambda_1 \leq \dots \leq \lambda_n$ the eigenvalues of the second fundamental form in the principal normal direction, that is $\nu_1= \frac{H}{|H|}$ (see Section 2 for precise definitions).   
\begin{theorem}
\label{thm:convex} 
Let $F:M\times[0,T) \to \mathbb{R}^{n+k}$ be a mean curvature flow of dimension $n \geq 5$ which is quadratically pinched with $c < c_n$. Then for any $\varepsilon >0$ there exists a constant $C_\varepsilon>0$ which depends only on $n$, $\varepsilon$ and $M_0$ such that 
\begin{align*}
\lambda_1 \geq -\varepsilon |H| - C_\varepsilon
\end{align*}
on $M_t$ for each $t \in [0,T)$. 
\end{theorem}
As $\varepsilon>0$ is arbitrary, this shows that the negative part of the first eigenvalue in the principal normal direction does not grow as fast as $|H|$. We then use this estimate to characterise type II singularities of the quadratically pinched mean curvature flow near the maximum of the curvature:
\begin{theorem}
\label{thm:sing}
Suppose a type II singularity forms at time $T$, that is 
\[\limsup_{t \nearrow T} \Big[(T-t)\max_{M_t} |A|^2 \Big] =\infty.\]
Then there exists a sequence of rescalings of $F$ that subconverges smoothly to a codimension one limiting flow which is either: a strictly convex translating solution; or the isometric product of $\mathbb{R}^m$ with a strictly convex translating solution of dimension $n -m$. 
\end{theorem}

The paper is set out as follows. In Section \ref{sec:evolution_equations} we gather together the necessary evolution equations and technical tools. In particular, $\lambda_1$ is not smooth but is locally Lipschitz and semiconvex, and its evolution equation must be understood in a distributional sense. In Section \ref{sec:Poinc} we obtain a Poincar\'e-type inequality which requires the Simons' identity for high codimension submanifolds. In Section \ref{sec:Stamp}, we complete the proof of the convexity estimate by applying Huisken's Stampacchia iteration. Finally, in Section \ref{sec:singularity_formation} we study singularity formation and prove Theorem \ref{thm:sing}.

\textbf{Acknowledgements.}  
The second named author was supported by the EPSRC grant EP/S012907/1.

\section{Evolution equations}

\label{sec:evolution_equations}

Let $F:M\times [0,T) \to \mathbb{R}^{n+k}$ solve mean curvature flow and write $M_t := F(M,t)$. We recall from the work of Andrews-Baker \cite{Andrews2010} the following evolution equations for the second fundamental form and mean curvature vector. With respect to local orthonormal frames $\{e_i\}$ and $\{\nu_\alpha\}$ for the tangent and normal bundles, 
\begin{align*}
\nabla_{\partial_t} A_{ij\alpha} & = \Delta A_{ij\alpha} + A_{ij\beta}A_{pq\beta} A_{pq\alpha} \\
&+A_{iq\beta}A_{qp\beta} A_{pj\alpha}+A_{jq\beta}A_{qp\beta} A_{pi\alpha} - 2 A_{ip\beta}A_{jq\beta}A_{pq\alpha} ,
\end{align*}
and 
\begin{align*}
\nabla_{\partial_t} H_{\alpha} & = \Delta H_{\alpha} + H_\beta A_{pq\beta} A_{pq \alpha}.
\end{align*}
From these equations we can compute that 
\begin{align*} 
(\partial_t -\Delta ) |A|^2 &= - 2 |\nabla A|^2 + 2 | \langle A, A\rangle |^2 + 2 |R^{\perp}|^2 \\
(\partial_t - \Delta) |H|^2 & = - 2 |\nabla H|^2 + 2 |\langle A, H\rangle |^2 .
\end{align*}
We use $R^{\perp}$ to denote the normal curvature, which is given by 
\begin{align*}
 R^\perp_{ij \alpha \beta} = A_{ip\alpha}  A_{jp\beta}- A_{jp \alpha} A_{ip\beta}.
\end{align*}

Under the quadratic pinching assumption we have $|H|>0$, so at any point in $M_t$ we can choose a local orthonormal frame for the normal bundle which is such that 
\[\nu_1 = \frac{H}{|H|}.\]
We also use the notation 
\[\hat A = A - \frac{1}{|H|} \langle A, H\rangle \frac{H}{|H|} = A - A_1 \nu_1\]
to denote the components of the second fundamental form orthogonal to the mean curvature vector, and write
\[h = \frac{1}{|H|}\langle A, H\rangle = A_1\]
for the scalar part of the mean curvature component of $A$. Hence $A$ admits the decomposition
\[A = h \nu_1 + \hat A.\]

\subsection{Pinching is preserved} With this notation in place we can state the estimate proven by Andrews-Baker showing that quadratic pinching is preserved by the flow. 
\begin{lemma}[\cite{Andrews2010}, Section 3]
\label{lem:pinch_pres}
Fix constants $0 <c < \frac{4}{3n}$ and $a >0$ and let \[\mathcal Q := |A|^2 - c|H|^2 + a.\] At every point in $M\times[0,T)$ where $\mathcal Q \leq 0$ there holds 
\begin{align}
\label{eqn_pinchpres2}
\nonumber(\partial_t-\Delta )  \mathcal Q & \leq -2 ( |\nabla A |^2-c |\nabla H|^2 )+2 |h|^2  \mathcal Q-2 a |h|^2-\frac{2a}{n} \frac{1}{c-\nicefrac{1}{n}}|\hat A|^2 \\ 
 &+\frac{2}{n}\frac{1}{c-\nicefrac{1}{n}}| \hat A|^2 \mathcal Q+ \left(6-\frac{2}{n (c-\nicefrac{1}{n})}  \right) |\circo h|^2 | \hat A|^2+\left(3-\frac{2}{n (c-\nicefrac{1}{n})}  \right)|\hat A|^4. 
\end{align}
\end{lemma}

Note that by Proposition 6 in \cite{Andrews2010} we have 
\begin{align}
\label{eq:grad_trace}
|\nabla A|^2 \geq \frac{3}{n+2} |\nabla H|^2
\end{align}
so the gradient term on the right-hand side is nonpositive. At points where $\mathcal Q\leq 0$, each of the zeroth-order reaction terms is also nonpositive. From now on we suppose the initial submanifold $M_0$, and hence $M_t$ for all $t \in [0,T)$, is quadratically pinched with 
\[c \leq \frac{4}{3n} - \varepsilon_0, \qquad \varepsilon >0.\]
For ease of notation let us define
\[W :=  \bigg( \frac{4}{3n} - \frac{\varepsilon_0}{2}\bigg) |H|^2 - |A|^2, \qquad w:= W^\frac{1}{2},\]
and observe that by the quadratic pinching $W \geq \frac{\varepsilon_0}{2} |H|^2$ on $M_t$. 

\begin{lemma}
\label{lem:w_evol}
At each point in $M \times [0,T)$ we have the inequalities 
\begin{align*}
(\partial_t - \Delta) W \geq 2 |h|^2  W + \frac{(n+2)}{3} \varepsilon_0 |\nabla A|^2.
\end{align*}
and 
\begin{align*}
 (\partial_t - \Delta)w \geq |h|^2 w +  \delta_0 \frac{|\nabla A|^2}{|H|},
\end{align*}
where $\delta_0 >0$ depends only on $n$ and $\varepsilon_0$.
\end{lemma}
\begin{proof}
From Lemma \ref{lem:pinch_pres} we obtain
\begin{align*}
(\partial_t - \Delta) W\geq 2 |h|^2 W + 2|\nabla A|^2 -2 \bigg( \frac{4}{3n} - \frac{\varepsilon_0}{2}\bigg)|\nabla H|^2.
\end{align*}
Using \eqref{eq:grad_trace} we estimate 
\begin{align*}
|\nabla A|^2 - \bigg( \frac{4}{3n} - \frac{\varepsilon_0}{2}\bigg)|\nabla H|^2 &= \bigg(1 - \frac{(n+2)}{3} \bigg( \frac{4}{3n} - \frac{\varepsilon_0}{2}\bigg)\bigg) |\nabla A|^2\\
&+\bigg( \frac{4}{3n} - \frac{\varepsilon_0}{2}\bigg) \bigg( \frac{n+2}{3} |\nabla A|^2 - |\nabla H|^2\bigg)\\
& \geq \bigg(1 - \frac{4(n+2)}{9n} \bigg) |\nabla A|^2 + \frac{(n+2)}{6} \varepsilon_0|\nabla A|^2\\
&\geq \frac{(n+2)}{6} \varepsilon_0|\nabla A|^2,
\end{align*}
which gives the desired inequality for $W$. It follows that 
\begin{align*}
(\partial_t - \Delta) W^\frac{1}{2}  &=  \frac{1}{4 W^{3/2} } |\nabla W|^2 + \frac{1}{2 W^\frac{1}{2} } (\partial_t - \Delta) W\\
&\geq |h|^2 W^\frac{1}{2} +  \frac{(n+2)}{6} \varepsilon_0 \frac{|\nabla A|^2}{W^\frac{1}{2}},
\end{align*}
and since $W \leq \frac{4}{3n} |H|^2$ we have
\begin{align*}
(\partial_t - \Delta) w &\geq |h|^2 w +  \frac{(3n)^\frac{1}{2}}{2}\frac{(n+2)}{6} \varepsilon_0 \frac{|\nabla A|^2}{|H|}.
\end{align*}
Thus it suffices to take 
\[\delta_0  = \frac{(3n)^\frac{1}{2}}{2}\frac{(n+2)}{6} \varepsilon_0 .\]
\end{proof}

\subsection{The evolution of $h$}

From the equations for $A$ and $H$, we readily compute the projection $\langle A, H\rangle $ satisfies 
\begin{align*}
(\partial_t - \Delta) A_{ij\alpha} H_\alpha &= -2\nabla_p A_{ij\alpha} \nabla_p H_\alpha + 2 H_\alpha A_{ij\beta}A_{pq\beta} A_{pq\alpha} \\
&+H_\alpha(A_{iq\beta}A_{qp\beta} A_{pj\alpha}+A_{jq\beta}A_{qp\beta} A_{pi\alpha} - 2 A_{ip\beta}A_{jq\beta}A_{pq\alpha} ).
\end{align*}
The first of the reaction terms can be split into a hypersurface and a codimension component, as follows:
\begin{align*}
2 H_\alpha A_{ij\beta}A_{pq\beta} A_{pq\alpha} &= 2 A_{ij\beta} A_{kl\beta}h_{kl} H_1 \\
&= 2 h_{ij}h_{pq}h_{pq}H_1 + 2 \sum_{\beta \neq 1} A_{ij\beta}A_{pq\beta} h_{pq} H_1 \\
&= 2  h_{ij} H_1 |h|^2 + 2 \sum_{\beta\neq1}^n \hat A_{ij\beta} \hat A_{pq\beta} h_{pq}H_1.
\end{align*}
Similarly, the remaining reaction terms can be written as 
\begin{align*}
H_\alpha(A_{iq\beta}A_{qp\beta} A_{pj\alpha}&+A_{jq\beta}A_{qp\beta} A_{pi\alpha} - 2 A_{ip\beta}A_{jq\beta}A_{pq\alpha} )\\
 & = h_{ip}h_{pq}h_{qj}H_1 + h_{jp}h_{pq}h_{qi}H_1 - 2 h_{ip}h_{jq}h_{pq}H_1\\
&+ \sum_{\beta \neq 1} A_{ip\beta} A_{pq \beta} h_{qj} H_1 + \sum_{\beta \neq 1}A_{jp\beta}A_{pq\beta}h_{qi}H_1 - 2 \sum_{\beta \neq 1} A_{ip\beta}A_{jq\beta}h_{pq}H_1 \\
&=   \sum_{\beta \neq 1} \hat A_{ip\beta} \hat A_{pq \beta} h_{qj} H_1 + \sum_{\beta \neq 1}\hat A_{jp\beta}\hat A_{pq\beta}h_{qi}H_1 - 2 \sum_{\beta \neq 1} \hat A_{ip\beta}\hat A_{jq\beta}h_{pq}H_1.
\end{align*}
Therefore,
\begin{align*}
(\partial_t - \Delta) A_{ij\alpha} H_\alpha &= -2\nabla_p A_{ij\alpha} \nabla_p H_\alpha +2|h|^2  h_{ij} H_1  + 2 \sum_{\beta\neq1}^n \hat A_{ij\beta} \hat A_{pq\beta} h_{pq}H_1\\
&+\sum_{\beta \neq 1} \hat A_{ip\beta} \hat A_{pq \beta} h_{qj} H_1 + \sum_{\beta \neq 1}\hat A_{jp\beta}\hat A_{pq\beta}h_{qi}H_1 - 2 \sum_{\beta \neq 1} \hat A_{ip\beta}\hat A_{jq\beta}h_{pq}H_1.
\end{align*}

For a positive function $f$, we have 
\begin{align*}
(\partial_t-\Delta) \sqrt{f} & =\frac{1}{4 f^{3/2} } |\nabla f|^2 + \frac{1}{2 \sqrt{f} }  (\partial_t - \Delta) f,
\end{align*}
hence the quantity $|H|$ satisfies 
\begin{align*} 
(\partial_t - \Delta) |H|  & =\frac{1}{4 |H|^3} |\nabla |H|^2|^2 + \frac{ 1}{2 |H| }(- 2 |\nabla H|^2 + 2 |\langle A, H\rangle |^2)\\
&=  \frac{|\langle A, H\rangle |^2 }{|H|} - \frac{|\nabla H|^2}{|H|}+ \frac{1}{|H|^3} \langle H, \nabla_i H\rangle  \langle H, \nabla_i H\rangle.
\end{align*}
There holds
\[\frac{|\langle A, H\rangle |^2 }{|H|} = |\langle A, |H|^{-1} H\rangle|^2  |H|= |h|^2 H_1\]
and 
\begin{align*}
- \frac{|\nabla H|^2}{|H|}+ \frac{1}{|H|^3}& \langle H, \nabla_i H\rangle  \langle H, \nabla_i H\rangle \\
&= -\frac{1}{H_1} ( H_1^2 |\nabla \nu_1|^2 + |\nabla H_1|^2) + \frac{1}{H_1} \langle \nu_1 , H_1 \nabla \nu_1 + \nabla H_1 \nu_1\rangle \\
& = - H_1 |\nabla \nu_1|^2
\end{align*}
so we have 
\begin{align} 
\label{eq:H_evol}
(\partial_t - \Delta) |H|  &  =  |h|^2 H_1 - H_1|\nabla \nu_1|^2. 
\end{align}

For a tensor $B_{ij}$ divided by a positive scalar function $f$ there holds 
\begin{align*}
(\nabla_{\partial_t} - \Delta) \frac{B_{ij}}{f} = \frac{1}{f} (\nabla_{\partial_t} - \Delta) B_{ij} - \frac{B_{ij}}{f^2} (\partial_t - \Delta) f + \frac{2}{f} \bigg\langle \nabla \frac{B_{ij}}{f}, \nabla f \bigg\rangle,
\end{align*}
Therefore, dividing $\langle A, H\rangle$ by $|H|$, we obtain 
\begin{align*}
(\nabla_{\partial_t} - \Delta) h_{ij} &=  |h|^2 h_{ij}  +2 H_1^{-1} \sum_{\beta\neq1}^n \hat A_{ij\beta} \hat A_{pq\beta} h_{pq}H_1 +H_1^{-1}\sum_{\beta \neq 1} \hat A_{ip\beta} \hat A_{pq \beta} h_{qj} H_1 \\
&+ H_1^{-1}\sum_{\beta \neq 1}\hat A_{jp\beta}\hat A_{pq\beta}h_{qi}H_1 - 2 H_1^{-1}\sum_{\beta \neq 1} \hat A_{ip\beta}\hat A_{jq\beta}h_{pq}H_1\\
&-2H_1^{-1} \langle \nabla A_{ij} ,\nabla H \rangle + h_{ij} |\nabla \nu_1|^2  +2 H_1^{-1} \langle \nabla h_{ij}, \nabla H_1 \rangle. 
\end{align*}
Let us introduce the abbreviation
\begin{align}
\label{eq:T_def}
T_{ij} &:= 2 \sum_{\beta\neq1}^n \hat A_{ij\beta} \hat A_{pq\beta} h_{pq}+\sum_{\beta \neq 1} \hat A_{ip\beta} \hat A_{pq \beta} h_{qj}  \notag \\
&+ \sum_{\beta \neq 1}\hat A_{jp\beta}\hat A_{pq\beta}h_{qi} - 2 \sum_{\beta \neq 1} \hat A_{ip\beta}\hat A_{jq\beta}h_{pq},
\end{align}
so that we may write 
\begin{align*}
(\nabla_{\partial_t} - \Delta) h_{ij} &=  |h|^2 h_{ij}  +T_{ij} -2H_1^{-1} \langle \nabla A_{ij}, \nabla H \rangle \\
&+ h_{ij} |\nabla \nu_1|^2 + 2 H_1^{-1} \langle \nabla h_{ij}, \nabla H_1 \rangle. 
\end{align*}
We simplify the gradient terms by decomposing
\begin{align*}
- 2\langle \nabla A_{ij} , \nabla H\rangle & = - 2\langle \nabla h_{ij} \nu_1 + h_{ij} \nabla \nu_1 + \nabla \hat A_{ij}, \nabla H_1 \nu_1 + 2H_1 \nabla \nu_1 \rangle\\
&= - 2\langle \nabla h_{ij} , \nabla H_1\rangle - 2H_1 h_{ij} |\nabla \nu_1|^2 - 2\langle \nabla \hat A_{ij}, \nabla H_1 \nu_1 \rangle \\
&- 2H_1 \langle \nabla \hat A_{ij}, \nabla \nu_1 \rangle, 
\end{align*} 
and so obtain:
\begin{lemma}
At each point in $M\times [0,T)$ there holds
\begin{align*}
(\nabla_{\partial_t} - \Delta) h_{ij} &=  |h|^2 h_{ij}  +T_{ij}-  h_{ij} |\nabla \nu_1|^2   - 2H_1^{-1} \langle \nabla \hat A_{ij}, \nabla H_1 \nu_1 \rangle\\
&- 2\langle \nabla \hat A_{ij}, \nabla \nu_1 \rangle ,
\end{align*}
where $T_{ij}$ is the quantity defined in \eqref{eq:T_def}. 
\end{lemma}

Since $h$ is a symmetric bilinear form it has $n$ real eigenvalues, which we denote by 
\[\lambda_1 \leq \dots \leq \lambda_n.\]
The smallest eigenvalue can be written as 
\[\lambda_1(x,t) = \min_{|v| = 1} h(x,t) (v,v),\]
and is therefore a Lipschitz continuous function on $M \times [0,T)$. We will use the evolution equation for $h$ to estimate $(\partial_t -\Delta)\lambda_1$, interpreted in an appropriate weak sense (cf. \cite{White} and \cite{Langford17}).

\begin{definition}
Let $f:M\times [0,T) \to \mathbb{R}$ be locally Lipschitz continuous and fix a point $(x_0,t_0) \in M \times (0,T)$. We say that a function $\varphi$ is a lower support for $f$ at $(x_0,t_0)$ if $\varphi$ is $C^2$ on the set $B_{g(t_0)}(x_0,r) \times [-r^2 +t_0,t_0]$ for some $r >0$ and there holds
\[f(x,t) \geq \varphi(x,t),\]
with equality at $(x_0,t_0)$. If the inequality is reversed then $\varphi$ is an upper support for $f$ at $(x_0,t_0)$. 
\end{definition}

With this definition in place we have the following estimate:

\begin{lemma}
\label{lem:lambda_1_evol_1} 
Fix $(x_0,t_0) \in M \times (0,T)$ and suppose $\varphi$ is a lower support for $\lambda_1$ at $(x_0,t_0)$. Then at $(x_0,t_0)$ there holds 
\begin{align*}
(\partial_t - \Delta) \varphi &\geq |h|^2 \varphi  +T_{11}-  \varphi |\nabla \nu_1|^2   - 2H_1^{-1} \langle \nabla \hat A_{11}, \nabla H_1 \nu_1 \rangle\\
&- 2\langle \nabla \hat A_{11}, \nabla \nu_1 \rangle
\end{align*}
\end{lemma}

\begin{proof}
We choose at the point $(x_0,t_0)$ an orthonormal basis of tangent vectors $\{e_i\}$ which are such that 
\[h(x_0,t_0) (e_i,e_i) = \lambda_i,\]
and extend the $\{e_i\}$ to a spatial neighbourhood of $x_0$ by parallel transport with respect to $g(t_0)$. We then extend to an orthonormal frame on a backward spacetime neighbourhood of $(x_0,t_0)$ by parallel transport with respect to the connection $\nabla_{\partial_t}$. On this neighbourhood we can define a smooth function 
\[\eta(x,t) := h(x,t)(e_1(x,t) ,e_1(x,t)).\]
Observe that by the definition of $\lambda_1$ there holds
\[\eta(x,t) \geq  \lambda_1(x,t) \geq \varphi(x,t),\]
with equality at $(x_0,t_0)$. 

It follows that at the point $(x_0,t_0)$ we have
\[\partial_t \eta \leq \partial_t \varphi, \qquad \Delta \eta \geq \Delta \varphi,\]
hence
\[(\partial_t - \Delta) \varphi \geq (\partial_t - \Delta) \eta.\]
At $(x_0,t_0)$ we compute
\begin{align*}
\partial_t \eta & = \nabla_{\partial_t} h(e_1, e_1) + 2 h(e_1, \nabla_{\partial_t} e_1)\\
& = \Delta h_{11} +|h|^2 h_{11}  +T_{11}-  h_{11} |\nabla \nu_1|^2   - 2H_1^{-1} \langle \nabla \hat A_{11}, \nabla H_1 \nu_1 \rangle\\
&- 2\langle \nabla \hat A_{11}, \nabla \nu_1 \rangle \\
& = \Delta h_{11} +|h|^2 \varphi  +T_{11}-  \varphi |\nabla \nu_1|^2   - 2H_1^{-1} \langle \nabla \hat A_{11}, \nabla H_1 \nu_1 \rangle\\
&- 2\langle \nabla \hat A_{11}, \nabla \nu_1 \rangle,
\end{align*}
and 
\begin{align*}
\Delta h_{11} &= \nabla_i (\nabla_i (h_{11}) - 2 h(e_1, \nabla_i e_1))\\
&=\Delta (h_{11}) - 2 h(e_1, \Delta e_1)\\
& = \Delta \eta - 2 h (e_1, \Delta e_1).
\end{align*}
On the other hand, at $(x_0,t_0)$ there holds 
\begin{align*}
\langle e_1, \Delta e_1\rangle = \nabla_k \langle e_1, \nabla_k e_1 \rangle =0,
\end{align*}
which shows that $\Delta e_1$ is orthogonal to $e_1$, so since $h$ is diagonal at $(x_0,t_0)$ we obtain
\begin{align*}
\Delta h_{11} = \Delta \eta,
\end{align*}
and consequently 
\begin{align*}
\partial_t \eta & = \Delta \eta +|h|^2 \varphi  +T_{11}-  \varphi |\nabla \nu_1|^2   - 2H_1^{-1} \langle \nabla \hat A_{11}, \nabla H_1 \nu_1 \rangle\\
&- 2\langle \nabla \hat A_{11}, \nabla \nu_1 \rangle.
\end{align*}
It follows then
\begin{align*}
(\partial_t - \Delta) \varphi &\geq (\partial_t -\Delta)\eta\\
		& = |h|^2 \varphi  +T_{11}-  \varphi |\nabla \nu_1|^2   - 2H_1^{-1} \langle \nabla \hat A_{11}, \nabla H_1 \nu_1 \rangle- 2\langle \nabla \hat A_{11}, \nabla \nu_1 \rangle
\end{align*}
at $(x_0,t_0)$ as required.
\end{proof}

Eventually we will want to prove integral estimates for the function $\lambda_1$. To do so we appeal to Alexandrov's theorem, following Brendle \cite{Brendle15} (see also \cite{Langford17}). We call a function $f:M\times[0,T) \to \mathbb{R}$ locally semiconvex (resp. semiconcave) if about every $(x_0,t_0)$ there is a small open neighbourhood on which $f$ can be expressed as the sum of a smooth and a convex (resp. concave) function. 

\begin{lemma}
\label{lem:alex}
Let $f:M\times[0,T) \to \mathbb{R}$ be locally semiconvex. Then $f$ is twice differentiable almost everywhere in $M\times[0,T)$, and if $\varphi$ is a nonnegative Lipschitz function on $M$ then for each $t \in [0,T)$ there holds
\[\int_M \Delta f \cdot \varphi \, d\mu_t \leq -\int_M \langle \nabla f, \nabla \varphi \rangle \,d\mu_t.\]
Here $\mu_t$ is the measure induced by the immersion $F(\cdot,t)$. 
\end{lemma}

\begin{proof}
Choosing local coordinates and applying Alexandrov's theorem \cite[Section 6.4]{Ev-Gar}, we see that $f$ has two derivatives at a.e. point in $M\times[0,T)$. Furthermore, by \cite[Section 6.3]{Ev-Gar}, for each $t \in [0,T)$ there is a singular Radon measure $\chi$ on $M$ with the property that
\[\int_M \Delta f \cdot \varphi  \,d\mu_t + \int_M \varphi \,d\chi = - \int_M \langle \nabla f, \nabla \varphi \rangle\,d\mu_t\]
for every $\varphi \in C^2(M)$ . Hence if $\varphi \geq 0$ there holds  
\[\int_M \Delta f \cdot \varphi  \,d\mu_t \leq  - \int_M \langle \nabla f, \nabla \varphi \rangle\,d\mu_t.\]
By approximation, the same inequality also holds if $\varphi$ is only Lipschitz continuous.
\end{proof}

Since $h$ is smooth, on every small enough set in spacetime, $\lambda_1$ can be expressed as the minimum over a set of smooth functions which is compact in $C^2$. This is sufficient to ensure that $\lambda_1$ is locally semiconcave on $M\times[0,T)$, so by the lemma we conclude that there is a set of full measure $Q \subset M\times [0,T)$ where $\lambda_1$ is twice differentiable. 

\begin{lemma}
\label{lem:lambda_1_evol_2}
At each point in $Q$ there holds 
\begin{align*}
(\partial_t - \Delta) \lambda_1 &\geq |h|^2 \lambda_1  +T_{11}-  \lambda_1 |\nabla \nu_1|^2   - 2H_1^{-1} \langle \nabla \hat A_{11}, \nabla H_1 \nu_1 \rangle\\
&- 2\langle \nabla \hat A_{11}, \nabla \nu_1 \rangle.
\end{align*}
\end{lemma}
\begin{proof}
Fix a point $(x_0,t_0)\in Q$. Then $\lambda_1$ admits a lower support $\varphi$ at $(x_0,t_0)$, to which we can apply Lemma \ref{lem:lambda_1_evol_1}. Since $\varphi(x_0,t_0) = \lambda_1(x_0,t_0)$, this gives the desired inequality.
\end{proof}

\begin{remark}
Notice that the first of the gradient terms is nonnegative whenever $\lambda_1 \leq 0$, whereas the remaining gradient terms both contain $\nabla \hat A$ as a factor. It is this structure of the gradient terms which allows us to prove the convexity estimate. 
\end{remark}

\subsection{The evolution of $|\hat A|^2$}

The following evolution equation for $|\hat A|^2$ was derived by Naff \cite{Naff2019}:
\begin{align*}
(\partial_t - \Delta) |\hat A|^2 &= 2 |\langle \hat A , \hat A\rangle |^2 + 2 \sum_{i,j} \bigg| \sum_k (\hat A_{ik} \otimes \hat A_{jk} - \hat A_{jk} \otimes \hat A_{ik} )\bigg|^2 + 2 \sum_\alpha |R^{\perp}_{ij1 \alpha}|^2\\
&- 2|\nabla \hat A|^2 + 4 \sum_{i,j,k} (\langle \nabla_k \circo h_{ij}, \nu_1 \rangle - H_1^{-1} \circo h_{ij} \nabla_k H_1 ) \langle \hat A_{ij}, \nabla_k \nu_1\rangle.
\end{align*}
We make use of the quantity 
\[v := \frac{|\hat A|^2}{|H|}.\]
\begin{lemma}
\label{lem:v_evol}
There is a positive constant $C = C(n)$ such that 
\begin{align*}
(\partial_t - \Delta) v &\leq C|A|^2 |\hat A| + C \frac{|\hat A|}{H_1} \frac{|\nabla A|^2}{H_1} - 2\frac{|\nabla \hat A|^2}{H_1} 
\end{align*}
holds on $M\times [0,T)$. 
\end{lemma}
\begin{proof}
We use the formula
\begin{align*}
(\partial_t - \Delta) \frac{f_1}{f_2} = \frac{1}{f_2} (\partial_t - \Delta) f_1 - \frac{f_1}{f_2^2} (\partial_t - \Delta) f_2 + \frac{2}{f_2} \bigg\langle \nabla \frac{f_1}{f_2}, \nabla f_2 \bigg\rangle
\end{align*}
to derive
\begin{align*}
(\partial_t - \Delta) v &= \frac{1}{|H|}  (\partial_t - \Delta) |\hat A|^2 - \frac{|\hat A|^2}{|H|^2} (\partial_t - \Delta) |H| + \frac{2}{|H|} \bigg\langle \nabla \frac{|\hat A|^2}{|H|}, \nabla |H|\bigg\rangle.
\end{align*}
Let us estimate 
\begin{align*}
\frac{2}{|H|} \bigg\langle \nabla \frac{|\hat A|^2}{|H|}, \nabla |H|\bigg\rangle & = \frac{2}{H_1} \bigg\langle \frac{1}{H_1} \nabla |\hat A|^2 - \frac{|\hat A|^2}{H_1^2} \nabla H_1, \nabla H_1\bigg\rangle \\
& = \frac{4}{H_1^2} \hat A_{ij}\langle  \nabla \hat A_{ij}, \nabla H_1\rangle  - 2 \frac{|\hat A|^2}{H_1^2} \frac{|\nabla H_1|^2}{H_1}\\
& \leq C(n) \frac{|\hat A|}{H_1} \frac{|\nabla A|^2}{H_1} ,
\end{align*}
where in the last line we have used $|A|^2 \leq \frac{4}{3n} |H|^2$. By \eqref{eq:H_evol} we have
\begin{align*} 
- \frac{|\hat A|^2}{|H|^2}(\partial_t - \Delta) |H|  &  = - \frac{|\hat A|^2}{H_1^2}( |h|^2 H_1 - H_1|\nabla \nu_1|^2) \\
&\leq \frac{|\hat A|^2}{H_1} \bigg| \nabla \frac{H}{H_1} \bigg|^2\\
&\leq C(n) \frac{|\hat A|}{H_1} \frac{|\nabla A|^2}{H_1},
\end{align*}
so there holds 
\begin{align}
\label{eq:v_evol_1}
(\partial_t - \Delta) v &\leq \frac{1}{H_1}  (\partial_t - \Delta) |\hat A|^2  + C(n) \frac{|\hat A|}{H_1} \frac{|\nabla A|^2}{H_1}.
\end{align}

We recall
\begin{align*}
(\partial_t - \Delta) |\hat A|^2 &= 2 |\langle \hat A , \hat A\rangle |^2 + 2 \sum_{i,j} \bigg| \sum_k (\hat A_{ik} \otimes \hat A_{jk} - \hat A_{jk} \otimes \hat A_{ik} )\bigg|^2 + 2 \sum_\alpha |R^{\perp}_{ij1 \alpha}|^2\\
&- 2|\nabla \hat A|^2 + 4 \sum_{i,j,k} (\langle \nabla_k \circo h_{ij}, \nu_1 \rangle - H_1^{-1} \circo h_{ij} \nabla_k H_1 ) \langle \hat A_{ij}, \nabla_k \nu_1\rangle,
\end{align*}
and estimate
\begin{align*}
(\partial_t - \Delta) |\hat A|^2 & \leq C(n)|\hat A|^4 + 2 \sum_\alpha |R^{\perp}_{ij1 \alpha}|^2 - 2|\nabla \hat A|^2 + C(n) \frac{|\hat A|}{H_1} |\nabla A|^2.
\end{align*}
Then since 
\begin{align*}
R^\perp_{ij\alpha \beta} = A_{ip\alpha} A_{jp\beta} - A_{ip\beta}A_{jp\alpha}
\end{align*}
we can write 
\begin{align*}
2 \sum_\alpha |R^{\perp}_{ij1 \alpha}|^2 = 2 \sum_{\alpha \geq 2} |h_{ip} \hat A_{jp\alpha} - \hat A_{ip\alpha} h_{jp}|^2
\end{align*}
and use this to bound
\begin{align*}
(\partial_t - \Delta) |\hat A|^2 & \leq C(n)|A|^2 |\hat A|^2 + C(n) \frac{|\hat A|}{H_1} |\nabla A|^2 - 2|\nabla \hat A|^2 .
\end{align*}
Substituting this inequality into \eqref{eq:v_evol_1} and using the quadratic pinching gives the desired estimate.
\end{proof}

\subsection{Modifying $\lambda_1$}

We now form the quantity 
\[f(x,t) := -\lambda_1(x,t) - \varepsilon w(x,t) + \Lambda v(x,t),\]
where $\varepsilon$ and $\Lambda$ are positive constants to be chosen later. Combining the evolution equations for the three components we obtain:

\begin{lemma}
\label{lem:f_evol}
At each point in $Q$ there holds 
\begin{align*}
(\partial_t - \Delta )f &\leq |h|^2 f + C(1+ \Lambda) |A|^2 |\hat A| - ( f + \varepsilon w )|\nabla \nu_1|^2 \\
&  - \bigg( \frac{\varepsilon \delta_0}{2}- C\Lambda  \frac{|\hat A|}{H_1} \bigg) \frac{|\nabla A|^2}{H_1} - \bigg(2\Lambda - \frac{C}{\varepsilon \delta_0}\bigg)\frac{|\nabla \hat A|^2}{H_1},
\end{align*}
where $C = C(n)$. 
\end{lemma}
\begin{proof}
At any point in $Q$ we compute 
\begin{align*}
(\partial_t - \Delta ) f = - (\partial_t - \Delta )\lambda_1 -\varepsilon(\partial_t - \Delta ) w+\Lambda (\partial_t - \Delta )v,
\end{align*}
so by Lemma \ref{lem:lambda_1_evol_2},
\begin{align*}
(\partial_t - \Delta )f &\leq -|h|^2 \lambda_1  -T_{11}+  \lambda_1 |\nabla \nu_1|^2   +2H_1^{-1} \langle \nabla \hat A_{11}, \nabla H_1 \nu_1 \rangle\\
&+ 2\langle \nabla \hat A_{11}, \nabla \nu_1 \rangle-\varepsilon(\partial_t - \Delta ) w+ \Lambda (\partial_t - \Delta )v.
\end{align*}
Inserting the estimates from Lemma \ref{lem:w_evol} and Lemma \ref{lem:v_evol} we find
\begin{align*}
(\partial_t - \Delta )f &\leq |h|^2 (-\lambda_1 - \varepsilon w)  -T_{11}+ \lambda_1 |\nabla \nu_1|^2   + 2H_1^{-1} \langle \nabla \hat A_{11}, \nabla H_1 \nu_1 \rangle\\
&+ 2\langle \nabla \hat A_{11}, \nabla \nu_1 \rangle-\varepsilon \delta_0 \frac{|\nabla A|^2}{H_1}+ \Lambda \bigg( C|A|^2 |\hat A| + C \frac{|\hat A|}{H_1} \frac{|\nabla A|^2}{H_1} - 2\frac{|\nabla \hat A|^2}{H_1}\bigg),
\end{align*}
where $C = C(n)$. Using the definition of $f$ and rearranging we obtain
\begin{align*}
(\partial_t - \Delta )f &\leq |h|^2 (f - \Lambda v)  -T_{11}+ C \Lambda |A|^2 |\hat A| + (- f - \varepsilon w + \Lambda v)|\nabla \nu_1|^2 \\
&  + 2H_1^{-1} \langle \nabla \hat A_{11}, \nabla H_1 \nu_1 \rangle+ 2\langle \nabla \hat A_{11}, \nabla \nu_1 \rangle + C\Lambda  \frac{|\hat A|}{H_1} \frac{|\nabla A|^2}{H_1}\\
&-\varepsilon \delta_0 \frac{|\nabla A|^2}{H_1}  - 2\Lambda\frac{|\nabla \hat A|^2}{H_1}.
\end{align*}
Next we estimate 
\begin{align*}
-T_{11} &= 2  \sum_{\beta\neq1}^n \hat A_{ij\beta} \hat A_{pq\beta} h_{pq} +\sum_{\beta \neq 1} \hat A_{ip\beta} \hat A_{pq \beta} h_{qj}  \\
&+\sum_{\beta \neq 1}\hat A_{jp\beta}\hat A_{pq\beta}h_{qi} - 2 \sum_{\beta \neq 1} \hat A_{ip\beta}\hat A_{jq\beta}h_{pq}\\
&\leq C(n) |A|^2 |\hat A|
\end{align*}
and
\begin{align*}
2H_1^{-1} \langle \nabla \hat A_{11}, \nabla H_1 \nu_1 \rangle + 2\langle \nabla \hat A_{11}, \nabla \nu_1 \rangle &\leq C(n) H_1^{-1} |\nabla \hat A| |\nabla A|\\
&\leq \frac{\varepsilon \delta_0}{2} \frac{|\nabla A|^2}{H_1} + \frac{C(n)}{\varepsilon \delta_0} \frac{|\nabla \hat A|^2}{H_1}
\end{align*}
in order to obtain
\begin{align*}
(\partial_t - \Delta )f &\leq |h|^2 (f - \Lambda v) + C(1+ \Lambda) |A|^2 |\hat A| - ( f + \varepsilon w -\Lambda v)|\nabla \nu_1|^2 \\
&  -\frac{\varepsilon \delta_0}{2} \frac{|\nabla A|^2}{H_1}  + C\Lambda  \frac{|\hat A|}{H_1} \frac{|\nabla A|^2}{H_1} + \bigg(\frac{C}{\varepsilon \delta_0} - 2\Lambda \bigg)\frac{|\nabla \hat A|^2}{H_1}.
\end{align*}
Finally, by bounding
\begin{align*}
\Lambda v |\nabla \nu_1|^2 \leq C(n) \Lambda \frac{|\hat A|}{H_1} \frac{|\nabla A|^2}{H_1} ,
\end{align*}
we find
\begin{align*}
(\partial_t - \Delta ) f &\leq |h|^2 (f-\Lambda v) + C(1+ \Lambda) |A|^2 |\hat A| - ( f + \varepsilon w )|\nabla \nu_1|^2 \\
&  - \bigg( \frac{\varepsilon \delta_0}{2}- C\Lambda  \frac{|\hat A|}{H_1} \bigg) \frac{|\nabla A|^2}{H_1} - \bigg(2\Lambda - \frac{C}{\varepsilon \delta_0}\bigg)\frac{|\nabla \hat A|^2}{H_1}.
\end{align*}
\end{proof}

\section{A Poincar\'{e} inequality}

\label{sec:Poinc}

In this section we establish a Poincar\'{e}-type inequality for the high codimension solution $M_t$. The proof loosely follows \cite[Lemma 5.4]{Huisken84}, in that we combine Simons' identity with an integration by parts argument. We also incorporate an idea from \cite[Proposition 3.1]{Bren-Huisk17}, where the authors symmetrise and then take the square of Simons' identity to fully exploit the structure of the cubic zeroth-order terms. 

Simons' identity for high codimension submanifolds states that
\begin{align*}
\nabla _k \nabla_ l A _{ij\alpha }  & = \nabla _i \nabla_j A_{kl \alpha } + A_{ kl \beta}  A _{ ip \beta} A _{ jp \alpha} - A_{ ij \beta} A _{ kp \beta} A _{ lp \alpha}\\
&+ A_{ jl \beta}  A_{ ip \beta} A _{ kp \alpha} + A_{ jk \beta} A _{ ip \beta} A _{ lp \alpha}-A _{ il \beta} A _{ kp \beta} A _{ jp \alpha}- A_{ jl \beta}A _{ kp \beta}A _{ ip \alpha}.
\end{align*}
We symmetrise to get 
\begin{align*}
\nabla _k \nabla_ l A _{ij\alpha } +\nabla_l\nabla_k  A_{ji\alpha}  =& \ \nabla _i \nabla_j A_{kl \alpha } + \nabla_j\nabla_i A_{lk\alpha} + E_{klij\alpha}.
\end{align*}
where
\begin{align*}
E_{klij\alpha} &= A_{ kl \beta}  A _{ ip \beta} A _{ jp \alpha} + A_{lk\beta} A_{jp\beta} A_{ip\alpha}\\
&- A_{ ij \beta} A _{ kp \beta} A _{ lp \alpha}-A_{ji\beta} A_{lp\beta}A_{kp\alpha}\\
&+ A_{ jl \beta}  A_{ ip \beta} A _{ kp \alpha} +A_{ik\beta} A_{jp\beta}A_{lp\alpha}\\
&+ A_{ jk \beta}  A _{ ip \beta} A _{ lp \alpha}+ A_{il\beta} A_{jp\beta}A_{kp\alpha}\\
&-A _{ il \beta} A _{ kp \beta} A _{ jp \alpha} -A_{jk\beta}A_{lp\beta}A_{ip\alpha}\\
&- A_{ jl \beta} A _{ kp \beta}A _{ ip \alpha}-A_{ik\beta} A_{lp\beta}A_{jp\alpha}.
\end{align*}
Using the relation
\begin{align*}
 R^\perp_{ij \alpha \beta}  =  A_{ip\alpha} A_{jp \beta} - A_{ip\beta} A_{jp\alpha}
\end{align*}
we can rewrite the components of $E$ as
\begin{align*}
E_{klij\alpha}=& \ A_{kl\beta} ( A_{ip\beta} A_{jp \alpha} + A_{jp\beta} A_{ip\alpha} ) - A_{ij \beta} (A_{kp\beta}A_{lp \alpha}+A_{lp\beta} A_{kp\alpha} ) \\
&+A_{jl\beta} (A_{ip\beta} A_{kp\alpha} - A_{kp\beta}A_{ip\alpha} ) + A_{jk\beta} ( A_{ip\beta}A_{lp\alpha} - A_{lp\beta}A_{ip\alpha } ) \\
&+A_{ik\beta} (A_{lp\alpha}A_{jp\beta} - A_{lp\beta} A_{jp\alpha} ) + A_{il\beta}( A_{kp\alpha} A_{jp\beta} - A_{kp\beta} A_{jp \alpha})\\
=& \ A_{kl\beta} ( A_{ip\beta} A_{jp \alpha} + A_{jp\beta} A_{ip\alpha} ) - A_{ij \beta} (A_{kp\beta}A_{lp \alpha}+A_{lp\beta} A_{kp\alpha} ) \\
& + A_{jl\beta} R^\perp_{ki\alpha\beta} + A _{ jk\beta} R^\perp _{li\alpha\beta} + A_{ik\beta} R^\perp_{lj\alpha\beta} + A _{il\beta}R^\perp _{ kj\alpha\beta}\\
= &  2 A_{kl\beta}  A_{ip\beta} A_{jp \alpha} - 2A_{ij \beta} A_{kp\beta}A_{lp \alpha}  + A_{kl\beta} R^\perp_{ij\alpha \beta} - A_{ij\beta} R^\perp_{kl\alpha\beta}\\
& + A_{jl\beta} R^\perp_{ki\alpha\beta} + A _{ jk\beta} R^\perp _{li\alpha\beta} + A_{ik\beta} R^\perp_{lj\alpha\beta} + A _{il\beta}R^\perp _{ kj\alpha\beta}.
\end{align*}

\begin{lemma}
There is a positive constant $C=C(n)$ such that 
\[|E|^2 \geq 8|h|^2 \tr (h^4) - 8 \tr(h^3)^2 - C|A|^5|\hat A|.\]
\end{lemma}
\begin{proof}
Let us decompose $E$ as 
\[E_{klij\alpha} = U_{klij\alpha} + V_{klij\alpha}\]
where 
\begin{align*}
U_{klij\alpha} &:= 2 A_{kl\beta}  A_{ip\beta} A_{jp \alpha} - 2A_{ij \beta} A_{kp\beta}A_{lp \alpha} ,\\
V_{klij\alpha} &:= A_{kl\beta} R^\perp_{ij\alpha \beta} - A_{ij\beta} R^\perp_{kl\alpha\beta} + A_{jl\beta} R^\perp_{ki\alpha\beta}\\
& + A _{ jk\beta} R^\perp _{li\alpha\beta} + A_{ik\beta} R^\perp_{lj\alpha\beta} + A _{il\beta}R^\perp _{ kj\alpha\beta}.
\end{align*}
There then holds 
\[|E|^2 = |U|^2 + 2 \langle U, V\rangle + |V|^2.\]
Breaking $U$ into components parallel and orthogonal to the mean curvature vector we obtain 
\begin{align*}
U_{klij} &= 2 \langle A_{kl} ,A_{ip} \rangle A_{jp} - 2\langle A_{ij}, A_{kp}\rangle A_{lp}\\
 & = 2 h_{kl} h_{ip}  A_{jp} - 2 h_{ij}h_{kp} A_{lp} + 2 \langle \hat A_{kl}, \hat A_{ip} \rangle A_{jp} - 2 \langle \hat A_{ij}, \hat A_{lp} \rangle A_{lp}\\
& =  2 h_{kl} h_{ip}  h_{jp} \nu_1  - 2 h_{ij}h_{kp} h_{lp} \nu_1+ 2 h_{kl} h_{ip}  \hat A_{jp} - 2 h_{ij}h_{kp} \hat A_{lp}\\
&+ 2 \langle \hat A_{kl}, \hat A_{ip} \rangle A_{jp} - 2 \langle \hat A_{ij}, \hat A_{lp} \rangle A_{lp},
\end{align*}
hence 
\begin{align*}
|U|^2 &\geq 4 | h_{kl} h_{ip}  h_{jp}  - h_{ij}h_{kp} h_{lp} |^2 - C(n) |\hat A| |A|^5+ 4 |\langle \hat A_{kl}, \hat A_{ip} \rangle A_{jp} -  \langle \hat A_{ij}, \hat A_{lp} \rangle A_{lp}|^2 \\
& \geq 8 |h|^2 \tr(h^4)-8\tr(h^3)^2 - C(n) |\hat A||A|^5.
\end{align*}
Substituting this back in we arrive at 
\begin{align*}
|E|^2 \geq 8 |h|^2 \tr(h^4)-8\tr(h^3)^2 - 2 |U||V| - C(n) |\hat A||A|^5. 
\end{align*}
There is a purely dimensional constant $C$ such that 
\[|V| \leq C|A||R^\perp|,\]
and we have
\[R^{\perp}_{ij1\beta} = h_{ip} \hat A_{jp\beta} - \hat A_{ip\beta} h_{jp}\]
and 
\[R^{\perp}_{ij\alpha\beta} = \hat A_{ip\alpha} A_{jp\beta} - A_{ip\beta} \hat A_{jp\alpha},\qquad \alpha \geq 2,\]
so for a larger constant $C$ there holds
\[|V| \leq C |A|^2 |\hat A|.\]
Since $|U| \leq C|A|^3$ we have 
\[|E|^2 \geq 8 |h|^2 \tr(h^4)-8\tr(h^3)^2 - C|A|^5 |\hat A|.\]
\end{proof}

We are now ready to prove the Poincar\'{e} inequality. The proof does not actually use the fact that $M_t$ moves by mean curvature, so this result can be viewed as a general statement about high codimension submanifolds.

\begin{proposition}
Fix $t \in [0,T)$ and let $u:M\to \mathbb{R}$ be a nonnegative Lipschitz function which is supported in $\supp(f(\cdot,t))$. Then there is a positive constant $C = C(n,\varepsilon_0,\varepsilon,\Lambda)$ such that 
\begin{align*}
\int_M |h|^2 u^2  \, d\mu_t & \leq C \int_M u^2  \frac{|\nabla A|^2}{|A|^2} \, d\mu_t +C\int_M u |\nabla u| \frac{|\nabla A|}{|A|} \,d \mu_t + C \int_M  |A| |\hat A| u^2\,d\mu_t.
\end{align*}
\end{proposition}

\begin{proof}
For a symmetric matrix $B$ with eigenvalues $\mu_i$ there holds
\begin{align*}
|B|^2 \tr (B^4) -  \tr(B^3)^2 &= \frac{1}{2} \sum_{i,j}( \mu_i^2  \mu_j^4 -  \mu_i^3 \mu_j^3 )+\frac{1}{2} \sum_{i,j}( \mu_j^2  \mu_i^4 -  \mu_i^3 \mu_j^3 )\\
&=\frac{1}{2}\sum_{i,j} \mu_i^2 \mu_j^2 (\mu_i^2 + \mu_j^2 - 2 \mu_i \mu_j)\\
&= \frac{1}{2}\sum_{i,j}\mu_i^2 \mu_j^2(\mu_i - \mu_j)^2.
\end{align*}
Observe that the right-hand side vanishes if and only if $B$ is the second fundamental form of a codimension-one cylinder. Let us define 
\[B(x,t) = h(x,t) - \Lambda v(x,t) g(x,t),\]
which has as eigenvalues $\mu_i = \lambda_i - \Lambda v$. In particular, the computation above shows that 
\begin{align*}
|B|^2 \tr (B^4) -  \tr(B^3)^2 &\geq  \mu_n^2 \mu_1^2 (\mu_n-\mu_1)^2\\
		&=\lambda_n^2 \mu_1^2 (\mu_n - \mu_1)^2 - 2 \Lambda \lambda_n  \mu_1^2 (\mu_n - \mu_1)^2 v + \Lambda^2 \mu_1^2 (\mu_n - \mu_1)^2 v^2\\
		&\geq \frac{1}{C} |h|^2 \mu_1^2 (\mu_n - \mu_1)^2 - C  |A|^5 |\hat A|
\end{align*}
where $C = C(n,\Lambda)$. At any point where $f(x,t) >0$ we have 
\[\lambda_1(x,t) < - \varepsilon w(x,t) + \Lambda v(x,t),\]
which is to say that $\mu_1(x,t) \leq - \varepsilon w(x,t)$.  Furthermore, since
\[0< H_1(x,t) = \lambda_1(x,t) + \dots + \lambda_n(x,t) \leq  \lambda_1(x,t) + (n-1)\lambda_n(x,t)\]
there holds 
\begin{align*}
\mu_n(x,t) - \mu_1(x,t) &= \lambda_n(x,t) - \lambda_1(x,t) \\
&\geq - \bigg(1+\frac{1}{n-1}\bigg) \lambda_1(x,t)\\
&\geq \frac{n}{n-1} \varepsilon w(x,t) - \frac{n}{n-1} \Lambda v(x,t).
\end{align*}
If the right-hand side is nonnegative then we can square both sides to get an estimate of the form 
\begin{align*}
(\mu_n(x,t) - \mu_1(x,t))^2 & \geq \frac{\varepsilon^2}{C} w(x,t)^2 - C\varepsilon \Lambda |A||\hat A|
\end{align*}
where $C = C(n)$. On the other hand if 
\[\frac{n}{n-1} \varepsilon w(x,t) - \frac{n}{n-1} \Lambda v(x,t)< 0 \]
then trivially there holds
\[(\mu_n(x,t) - \mu_1(x,t))^2 \geq 0 \geq \frac{n^2}{(n-1)^2} \varepsilon^2 w(x,t)^2 - \frac{n^2}{(n-1)^2} \Lambda^2 v(x,t)^2,\]
so in either case we can bound
\begin{align*}
(\mu_n(x,t) - \mu_1(x,t))^2 & \geq \frac{\varepsilon^2}{C} w(x,t)^2 - C |A||\hat A|
\end{align*}
with $C = C(n,\varepsilon, \Lambda)$. 

Putting these estimates together we find on the support of $f$ there holds
\begin{align*}
|B|^2 \tr (B^4) -  \tr(B^3)^2 &\geq  \mu_n^2 \mu_1^2 (\mu_n-\mu_1)^2\\
		&\geq \frac{\varepsilon^4}{C} |h|^2  w^4  - C |A|^5 |\hat A|,
\end{align*}
and since 
\[|B|^2 \tr (B^4) -  \tr(B^3)^2 \leq |h|^2 \tr(h^4) - \tr(h^3)^2 + C |A|^5 |\hat A|\]
we finally get 
\[|h|^2 \tr(h^4) - \tr(h^3)^2 \geq C^{-1}  |h|^2 |A|^4 - C | A|^5||\hat A| \]
where $C = C(n, \varepsilon_0, \varepsilon, \Lambda)$. 

Combining this inequality with the result of the last lemma we find on the support of $f$ there holds
\begin{align*}
|h|^2 |A|^4  &\leq C(|h|^2 \tr(h^4) - \tr(h^3)^2 )+ C| A|^5||\hat A|\\
& \leq C|E|^2 + C|A|^5 |\hat A|.
\end{align*}
Let $u$ be a nonnegative Lipschitz function supported in $\supp(f)$. Then we can multiply this inequality by $|A|^{-4} u^2$ and integrate over $M$ to get 
\begin{align*}
\int_M & |h|^2 u^2 \,d\mu_t \\
&\leq C \int_M |A|^{-4} u^2  |E|^2  + |A| |\hat A| u^2 \, d\mu_t\\
& = C \int_M |A|^{-4} u^2 E_{klij\alpha} (\nabla _k \nabla_ l A _{ij\alpha } +\nabla_l\nabla_k  A_{ji\alpha}  -\nabla _i \nabla_j A_{kl \alpha } - \nabla_j\nabla_i A_{lk\alpha}) \, d\mu_t\\
& + C \int_M |A| |\hat A| u^2 \,d\mu_t
\end{align*}
We are going to estimate each of the four Hessian terms on the right. Since each of these is handled in the same way, we only give the argument for the first one. Defining
\[T_k :=  |A|^{-4} u^2 E_{ijkl\alpha} \nabla_ l A _{ij\alpha },\]
we can write 
\begin{align*}
|A|^{-4} u^2 E_{ijkl\alpha} \nabla _k \nabla_ l A _{ij\alpha } & = \nabla_k T_k  + 4 |A|^{-5}   u^2 E_{ijkl\alpha} \nabla_ l A _{ij\alpha } \nabla_k |A| \\
&-2 |A|^{-4} u E_{ijkl\alpha}\nabla_ l A _{ij\alpha } \nabla_k u  -  |A|^{-4} u^2 \nabla_k E_{ijkl\alpha} \nabla_ l A _{ij\alpha }.
\end{align*}
The divergence term vanishes upon integration, and there is a purely dimensional constant $C$ such that 
\[|E| \leq C|A|^3 , \qquad |\nabla E| \leq C |A|^2 |\nabla A|, \qquad |\nabla |A|| \leq C |\nabla A|,\]
so making $C$ a bit larger, we have 
\begin{align*}
\int_M |A|^{-4} u^2   E_{ijkl\alpha}\nabla _k \nabla_ l A _{ij\alpha }  \, d\mu_t &\leq C \int_M u^2 |A|^{-5} |A|^3 |\nabla A|^2 \, d\mu_t\\
& + C \int_M u |\nabla u| |A|^{-4} |A|^3 |\nabla A| \,d \mu_t\\
& + C\int_M u^2 |A|^{-4} |A|^2 |\nabla A|^2 \,d\mu_t.
\end{align*}
Estimating the remaining Hessian terms in the same way and substituting back in we arrive at 
\begin{align*}
\int_M |h|^2 u^2  \, d\mu_t & \leq C \int_M u^2  \frac{|\nabla A|^2}{|A|^2} \, d\mu_t +C\int_M u |\nabla u| \frac{|\nabla A|}{|A|} \,d \mu_t + C \int_M  |A| |\hat A| u^2\,d\mu_t.
\end{align*}
\end{proof}

\section{Stampacchia iteration}

\label{sec:Stamp}

In this section we establish the convexity estimate by proving an a priori supremum estimate for the function 
\[f_\sigma := \frac{f}{|H|^{1-\sigma}}\]
where $\sigma \in (0,1)$ is chosen small depending on $n$ and $M_0$. Recall from Lemma \ref{lem:f_evol} that at each point in $Q$ there holds 
\begin{align*}
(\partial_t - \Delta )f &\leq |h|^2 f + C(1+ \Lambda) |A|^2 |\hat A| - ( f + \varepsilon w )|\nabla \nu_1|^2 \\
&  - \bigg( \frac{\varepsilon \delta_0}{2}- C\Lambda  \frac{|\hat A|}{H_1} \bigg) \frac{|\nabla A|^2}{H_1} - \bigg(2\Lambda - \frac{C}{\varepsilon \delta_0}\bigg)\frac{|\nabla \hat A|^2}{H_1},
\end{align*}
where $C = C(n)$. Let us fix 
\[\Lambda = \frac{C}{2\varepsilon \delta_0}\]
so that the last term vanishes. Then using 
\begin{align*}
(\partial_t - \Delta) |H|^{1-\sigma} = (1-\sigma) |h|^2 H_1^{1-\sigma} - (1-\sigma)H_1^{1-\sigma}|\nabla \nu_1|^2 + \sigma (1-\sigma) H_1^{-\sigma - 1} |\nabla H_1|^2
\end{align*}
we compute that
\begin{align*}
(\partial_t - \Delta) f_\sigma &\leq \sigma |h|^2 f_\sigma + C(1+ \Lambda) |A|^2 \frac{|\hat A|}{H_1^{1-\sigma}}  - \bigg(\sigma f_\sigma + \varepsilon \frac{w}{H_1^{1-\sigma} }\bigg)|\nabla \nu_1|^2 \\
&  - \bigg( \frac{\varepsilon \delta_0}{2}- C\Lambda  \frac{|\hat A|}{H_1} \bigg) H^\sigma \frac{|\nabla A|^2}{H_1^2}   - \sigma(1-\sigma) f_\sigma \frac{ |\nabla H_1|^2}{H_1^2}\\
& + 2(1-\sigma) \bigg\langle \nabla f_\sigma, \frac{\nabla H_1}{H_1} \bigg\rangle.
\end{align*}
Hence at points in $Q \cap \supp(f_\sigma)$ we have 
\begin{align}
\label{eq:f_sigma_evol}
(\partial_t - \Delta) f_\sigma &\leq \sigma |h|^2  f_\sigma+C(1+ \Lambda) |A|^2 \frac{|\hat A|}{H_1^{1-\sigma}}- \bigg( \frac{\varepsilon \delta_0}{2}- C\Lambda  \frac{|\hat A|}{H_1} \bigg) H_1^\sigma\frac{|\nabla A|^2}{H_1^{2}} \notag\\
& + 2(1-\sigma) \bigg\langle \nabla  f_\sigma, \frac{\nabla H_1}{H_1} \bigg\rangle,
\end{align}
where $C=C(n)$. 

All of the computations until now were for a quadratically pinched solution with 
\[c \leq \frac{4}{3n} - \varepsilon_0.\]
From here on we assume $n\geq 5$ and the more restrictive condition $c \leq c_n - \varepsilon_0$
where 
\[c_n : =\begin{cases} 
      \frac{3(n+1)}{2n(n+2)} & n = 5, 6, 7, \\
      \frac{4}{3n} & n \geq 8.
   \end{cases}
\]
This is the range of pinching constants for which Naff's codimension estimate is valid.
\begin{theorem}[\cite{Naff2019}]
\label{thm:codim}
Let $F:M\times[0,T) \to \mathbb{R}^{n+k}$, $n \geq 5$, be a quadratically pinched mean curvature flow with $c \leq c_n - \varepsilon_0$. Then there is a constant $\eta = \eta(n,\varepsilon_0)$ in $(0,1)$ such that 
\[\max_{M_t} \frac{|\hat A|^2}{|H|^{2-2\eta}} \leq \max_{M_0} \frac{|\hat A|^2}{|H|^{2-2\eta}}\]
for each $t \in [0,T)$. 
\end{theorem}
Hence if we set 
\[L := \max_{M_0} |H|\]
then the inequality
\[\frac{|\hat A|}{|H|} \leq C(n)L^\eta |H|^{-\eta}\]
holds on $M_t$ for every $t \in [0,T)$. Inserting this estimate into \eqref{eq:f_sigma_evol} we find
\begin{align}
\label{eq:f_sigma_est_1}
(\partial_t - \Delta) f_\sigma &\leq \sigma |h|^2  f_\sigma+C(1+ \Lambda)L^\eta |A|^2 H_1^{\sigma-\eta}- \bigg( \frac{\varepsilon \delta_0}{2}- C\Lambda L^\eta H_1^{-\eta} \bigg) H_1^\sigma\frac{|\nabla A|^2}{H_1^{2}}\notag\\
& + 2(1-\sigma) \bigg\langle \nabla  f_\sigma, \frac{\nabla H_1}{H_1} \bigg\rangle
\end{align}
on $Q \cap \supp(f_\sigma)$, where $C=C(n)$. 

\subsection{$L^p$-estimates}
For each $k >0$ let us define 
\[ f_{\sigma,k}(x,t) := \max\{ f_\sigma(x,t) -k,0\}.\]
Using the Poincar\'{e} inequality we now establish an $L^p$-estimate for $f_{\sigma,k}$. In the codimension one case similar estimates have appeared in \cite{Huisken84} and \cite{Huisk-Sin99a}.

\begin{proposition}
\label{prop:Lp}
There are positive constants $p_0$ and $\ell_0$ depending on $n$, $\varepsilon_0$, $\eta$, $\varepsilon$ and $\Lambda$, and a positive constant $k_0 = k_0(n, \varepsilon_0, \eta, \varepsilon, \Lambda, L)$, with the following property. For every 
\[p \geq p_0, \qquad \sigma \leq \ell_0 p^{-\frac{1}{2}}, \qquad k \geq k_0,\]
we have 
\[\sup_{t \in [0,T)} \bigg( \int_M f_{\sigma,k}^p \,d\mu_t \bigg)\leq C,\]
where $C = C(n, \varepsilon_0,\eta,\varepsilon, \Lambda, L, \mu_0(M), T, k, \sigma, p)$.  
\end{proposition}
\begin{proof}
Suppose for now that $p_0 \geq 4$ and $\ell_0 \leq \eta$. Then the condition $\sigma \leq \ell_0 p^{-\frac{1}{2}}$ ensures that $\sigma \leq \eta /2$. On $\supp(f_{\sigma,k})$ we have 
\[k < \frac{f}{H_1} H_1^\sigma \leq C_0(n,\Lambda) H_1^\sigma,\]
so if we take $k_0 \geq C_0$ and impose $k \geq k_0$ then on $\supp(f_{\sigma,k})$ there holds
\[H_1 \geq (k/C_0)^\frac{1}{\sigma} \geq \max\{k/C_0,1\}.\]
Substituting this into \eqref{eq:f_sigma_est_1} we find
\begin{align*}
(\partial_t - \Delta) f_\sigma &\leq \sigma |h|^2  f_\sigma+C(1+ \Lambda) L^\eta |A|^2 H_1^{-\frac{\eta}{2}}- \bigg( \frac{\varepsilon \delta_0}{2}- C_1 k^{-\eta} \bigg) H_1^\sigma\frac{|\nabla A|^2}{H_1^{2}}\notag\\
& + 2(1-\sigma) \bigg\langle \nabla  f_\sigma, \frac{\nabla H_1}{H_1} \bigg\rangle
\end{align*}
on $Q\cap \supp(f_{\sigma, k})$, where $C = C(n)$ and $C_1 = C_1(n,\eta,\Lambda,L)$. Choosing $k_0$ a bit larger so that 
\[k_0 \geq \max\bigg\{1,C_0, \bigg( \frac{4C_1}{\varepsilon \delta_0}\bigg)^{1/\eta} \bigg\}\]
and using $f/H_1 \leq C_0$, we find on $Q \cap \supp(f_{\sigma,k})$,
 \begin{align*}
(\partial_t - \Delta) f_\sigma &\leq \sigma |h|^2  f_\sigma+C(1+ \Lambda)L^\eta |A|^2 H_1^{-\frac{\eta}{2}}-  \frac{\varepsilon \delta_0}{4 C_0} f_\sigma \frac{|\nabla A|^2}{H_1^{2}}\notag\\
& + 2(1-\sigma) \bigg\langle \nabla  f_\sigma, \frac{\nabla H_1}{H_1} \bigg\rangle.
\end{align*}
By Young's inequality we have
\[2(1-\sigma) \bigg\langle \nabla  f_\sigma, \frac{\nabla H_1}{H_1} \bigg\rangle \leq C_2 \frac{|\nabla f_\sigma|^2}{f_\sigma} + \frac{\varepsilon \delta_0}{8 C_0} f_\sigma \frac{|\nabla A|^2}{H_1^2}\]
on $\supp(f_\sigma)$, where $C_2 = C_2(n,\varepsilon_0,\varepsilon, C_0)$. Hence on $Q \cap \supp(f_{\sigma,k})$, 
 \begin{align*}
(\partial_t - \Delta) f_\sigma &\leq \sigma |h|^2  f_\sigma+C(1+ \Lambda) L^\eta |A|^2 H_1^{-\frac{\eta}{2}}-  \frac{\varepsilon \delta_0}{8 C_0} f_\sigma \frac{|\nabla A|^2}{H_1^{2}} + C_2 \frac{|\nabla f_\sigma|^2}{f_\sigma}.
\end{align*}

Applying the pinching we can bound
\[C(1+ \Lambda) L^\eta |A|^2 H_1^{-\eta/2} \leq C_3(n,\eta,\Lambda,L) |h|^2 H_1^{-\eta/2},\]
and by Young's inequality 
\[H_1^{-\eta/2} \leq \frac{4 - \eta}{4}s^{4/(4-\eta)} + \frac{\eta}{4} \frac{1}{s^{4/\eta} } \frac{1}{H_1^2} \leq s^{4/(4-\eta)} + \frac{1}{s^{4/\eta} } \frac{1}{H_1^2} \]
for every positive $s$. Setting $s = \sigma^{(4-\eta)/4}$ gives 
\[H_1^{-\eta/2} \leq \sigma + \frac{1}{\sigma^{(4-\eta)/\eta}} \frac{1}{H_1^2} \leq \sigma + \sigma^{-4/\eta} H_1^{-2},\]
so using the pinching we get
\[C(1+ \Lambda) L^\eta|A|^2 H_1^{-\frac{\eta}{2}} \leq C_3 \sigma |h|^2 + C_4 \sigma^{-4/\eta}\]
for some $C_4 = C_4(n,\eta,\Lambda,L)$. Substituting back in, we have 
 \begin{align*}
(\partial_t - \Delta) f_\sigma &\leq \sigma |h|^2  f_\sigma + C_3\sigma |h|^2-  c_0 f_\sigma \frac{|\nabla A|^2}{H_1^{2}} + C_2 \frac{|\nabla f_\sigma|^2}{f_\sigma} + C_4 \sigma^{-4/\eta}
\end{align*}
on $Q \cap \supp (f_{\sigma,k})$,
where 
\[c_0 := \frac{\varepsilon \delta_0}{8 C_0}.\] 

If $\varphi$ is any nonnegative Lipschitz function supported in $\supp(f_{\sigma,k})$, then on almost every timeslice we can multiply the last inequality by $\varphi$ and integrate to get 
 \begin{align*}
\int_M \partial_t f_\sigma \cdot \varphi \,d\mu_t &\leq \int_M \Delta f_\sigma \cdot \varphi\,d\mu_t + \sigma \int_M |h|^2 f_\sigma \varphi \,d\mu_t + C_3 \sigma \int_M |h|^2 \varphi\, d\mu_t \\
&- c_0 \int_M f_\sigma \varphi \frac{|\nabla A|^2}{|H|^2} \,d\mu_t + C_2 \int_M \varphi \frac{|\nabla f_\sigma|^2}{f_\sigma}\,d\mu_t + C_4 \sigma^{-4/\eta}\int_M \varphi\,d\mu_t.
\end{align*}
Since $f_\sigma$ is a locally semiconvex function we can use Lemma \ref{lem:alex} to integrate by parts, and so obtain
 \begin{align*}
\int_M \partial_t f_\sigma \cdot \varphi \,d\mu_t & \leq -\int_M \langle \nabla f_\sigma, \nabla \varphi \rangle\,d\mu_t + \sigma \int_M |h|^2 f_\sigma \varphi \,d\mu_t + C_3 \sigma \int_M |h|^2 \varphi\, d\mu_t \\
&- c_0 \int_M f_\sigma \varphi \frac{|\nabla A|^2}{|H|^2} \,d\mu_t + C_2 \int_M \varphi \frac{|\nabla f_\sigma|^2}{f_\sigma}\,d\mu_t + C_4 \sigma^{-4/\eta}\int_M \varphi\,d\mu_t.
\end{align*}
We set $\varphi = p f_{\sigma, k}^{p-1}$ in this inequality and use 
\[\frac{d}{dt} \int_M f_{\sigma, k}^p \,d\mu_t = p \int_M \partial_t f_\sigma \cdot f_{\sigma,k}^{p-1} \,d\mu_t - \int_M |H|^2 f_{\sigma,k}^p\,d\mu_t\]
to estimate
\begin{align*}
\frac{d}{dt} &\int_M f_{\sigma,k}^p \,d\mu_t\\
 &\leq - p(p-1) \int_M f_{\sigma,k}^{p-2}|\nabla f_\sigma|^2\,d\mu_t + \sigma p\int_M |h|^2 f_\sigma f_{\sigma,k}^{p-1} \,d\mu_t + C_3 \sigma p \int_M |h|^2 f_{\sigma,k}^{p-1}\, d\mu_t \\
&- c_0 p\int_M f_\sigma f_{\sigma,k}^{p-1} \frac{|\nabla A|^2}{|H|^2} \,d\mu_t + C_2 p\int_M f_{\sigma,k}^{p-1} \frac{|\nabla f_\sigma|^2}{f_\sigma}\,d\mu_t + C_4 \sigma^{-4/\eta} p \int_M f_{\sigma,k}^{p-1} \,d\mu_t
\end{align*}
for almost every $t \in [0,T)$. Using that $f_{\sigma,k} = f_\sigma - k$ on $\supp(f_{\sigma,k})$ and rearranging slightly, this gives 
\begin{align*}
\frac{d}{dt} \int_M f_{\sigma,k}^p \,d\mu_t &\leq - (p(p-1) - C_2p)\int_M f_{\sigma,k}^{p-2}|\nabla f_\sigma|^2\,d\mu_t - c_0 p\int_M f_{\sigma,k}^{p} \frac{|\nabla A|^2}{|H|^2} \,d\mu_t \\
 & + \sigma p\int_M |h|^2 f_{\sigma,k}^{p} \,d\mu_t  + (C_3+k) \sigma p \int_M |h|^2 f_{\sigma,k}^{p-1}\, d\mu_t \\
 &+ C_4 \sigma^{-4/\eta} p \int_M f_{\sigma,k}^{p-1} \,d\mu_t.
\end{align*}

Using Young's inequality we estimate 
\begin{align*}
(C_3 +k)\sigma p\int_M |h|^2  f_{\sigma,k}^{p-1}\, d\mu_t  & \leq \sigma (p-1) \int_M |h|^2  f_{\sigma,k}^p \,d\mu_t + (C_3+k)^p \sigma \int_M |h|^2 \, d\mu_t
\end{align*}
and
\begin{align*}
C_4 \sigma^{-4/\eta} p \int_M f_{\sigma,k}^{p-1} \,d\mu_t \leq C_4 \sigma^{-4/\eta} (p-1) \int_M f_{\sigma,k}^{p} \,d\mu_t + C_4 \sigma^{-4/\eta} \mu_t(M).
\end{align*}
Inserting these inequalities we arrive at 
\begin{align}
\label{eq:Lp_step}
\frac{d}{dt} \int_M f_{\sigma,k}^p \,d\mu_t &\leq - (p(p-1) - C_2p)\int_M f_{\sigma,k}^{p-2}|\nabla f_\sigma|^2\,d\mu_t - c_0 p\int_M f_{\sigma,k}^{p} \frac{|\nabla A|^2}{|H|^2} \,d\mu_t \notag\\
 & + 2\sigma p\int_M |h|^2 f_{\sigma,k}^{p} \,d\mu_t + (C_3+k)^p \sigma \int_M |h|^2 \, d\mu_t   \notag\\
&+C_4 \sigma^{-4/\eta} p \int_M f_{\sigma,k}^{p} \,d\mu_t + C_4 \sigma^{-4/\eta} \mu_t(M).
\end{align}

Since $ f_{\sigma, k}$ is supported in $\supp(f)$, we can apply the Poincar\'{e} inequality with $u = f_{\sigma, k}^{\frac{p}{2}}$ to obtain
\begin{align*}
\int_M |h|^2 f_{\sigma,k}^p  \, d\mu_t & \leq C_5 \int_M  f_{\sigma,k}^p  \frac{|\nabla A|^2}{|H|^2} \, d\mu_t +C_5 p \int_M  f_{\sigma,k}^{p-1} |\nabla  f_\sigma| \frac{|\nabla A|}{|H|} \,d \mu_t\\
& + C_5 \int_M  |A| |\hat A| f_{\sigma,k}^p\,d\mu_t,
\end{align*} 
where the constant $C_5$ depends on $n$, $\varepsilon_0$, $\varepsilon$ and $\Lambda$. Applying Young's inequality we obtain
\begin{align*}
\int_M |h|^2  f_{\sigma,k}^p  \, d\mu_t & \leq C_5(1+p^\frac{1}{2}) \int_M  f_{\sigma,k}^p  \frac{|\nabla A|^2}{|H|^2} \, d\mu_t +C_5 p^\frac{3}{2} \int_M   f_{\sigma,k}^{p-2} |\nabla  f_\sigma|^2 \,d \mu_t\\
& + C_5 \int_M  |A| |\hat A| f_{\sigma,k}^p\,d\mu_t.
\end{align*}
Inserting the codimension estimate and quadratic pinching we get 
\begin{align*}
C_5 \int_M  |A| |\hat A| f_{\sigma,k}^p\,d\mu_t \leq C_6(n, L, C_5) \int_M |h|^2 |H|^{-\eta} f_{\sigma,k}^p\,d\mu_t, 
\end{align*}
and we know that $|H| \geq k/C_0$ on $\supp(f_{\sigma,k})$, so if we take 
\[ k_0 \geq \max\bigg\{1,C_0, \bigg( \frac{4C_1}{\varepsilon \delta_0}\bigg)^{1/\eta} , C_0(2C_6)^{1/\eta}\bigg\}\]
then 
\begin{align*}
C_5 \int_M  |A| |\hat A| f_{\sigma,k}^p\,d\mu_t \leq \frac{1}{2} \int_M |h|^2  f_{\sigma,k}^p\,d\mu_t. 
\end{align*}
In this case 
\begin{align*}
\frac{1}{2} \int_M |h|^2  f_{\sigma,k}^p  \, d\mu_t & \leq C_5(1+p^\frac{1}{2}) \int_M  f_{\sigma,k}^p  \frac{|\nabla A|^2}{|H|^2} \, d\mu_t +C_5 p^\frac{3}{2} \int_M   f_{\sigma,k}^{p-2} |\nabla  f_\sigma|^2 \,d \mu_t.
\end{align*}
Multiplying this inequality through by $4\sigma p$ and substituting back into \eqref{eq:Lp_step} gives
\begin{align*}
\frac{d}{dt} \int_M f_{\sigma,k}^{p} \,d\mu_t &\leq- (p(p-1)-C_2p - 4 C_5 \sigma p^\frac{5}{2}) \int_M  f_{\sigma, k}^{p-2} |\nabla f_\sigma|^2 \,d\mu_t \\
&- (c_0 p - 4 C_5\sigma p -4 C_5\sigma p^\frac{3}{2} ) \int_M f_{\sigma,k}^{p} \frac{|\nabla A|^2}{|H|^{2}}\,d\mu_t\\
&+ (C_3+k)^p \sigma \int_M |h|^2 \, d\mu_t  +C_4 \sigma^{-4/\eta} p \int_M f_{\sigma,k}^{p} \,d\mu_t \\
&+ C_4 \sigma^{-4/\eta} \mu_t(M).
\end{align*}

Now we insert the assumption $\sigma \leq \ell_0 p^{-\frac{1}{2}}$ and thus obtain
\begin{align*}
\frac{d}{dt} \int_M  f_{\sigma,k}^{p} \,d\mu_t &\leq- (p(p-1)-C_2p -4 C_5 \ell_0 p^2) \int_M f_{\sigma, k}^{p-2} |\nabla  f_\sigma|^2 \,d\mu_t \\
&- (c_0 p - 4 C_5\ell_0 p^\frac{1}{2} -4 C_5 \ell_0 p ) \int_M f_{\sigma,k}^{p} \frac{|\nabla A|^2}{|H|^{2}}\,d\mu_t\\
&+ (C_3+k)^p \sigma \int_M |h|^2 \, d\mu_t  +C_4 \sigma^{-4/\eta} p \int_M f_{\sigma,k}^{p} \,d\mu_t \\
&+ C_4 \sigma^{-4/\eta} \mu_t(M).
\end{align*}
Decreasing $\ell_0$ so that
\[\ell_0 \leq \min\bigg\{\eta, \frac{c_0}{8C_5}, \frac{1}{8C_5} \bigg\}\]
now gives
\begin{align*}
\frac{d}{dt} \int_M  f_{\sigma,k}^{p} \,d\mu_t &\leq- (p^2/2 -p-C_2p ) \int_M  f_{\sigma, k}^{p-2} |\nabla f_\sigma|^2 \,d\mu_t \\
&- (c_0 p /2 - 2 C_5\ell_0 p^\frac{1}{2}  ) \int_M f_{\sigma,k}^{p} \frac{|\nabla A|^2}{|H|^{2}}\,d\mu_t\\
&+ (C_3+k)^p \sigma \int_M |h|^2 \, d\mu_t  +C_4 \sigma^{-4/\eta} p \int_M f_{\sigma,k}^{p} \,d\mu_t \\
&+ C_4 \sigma^{-4/\eta} \mu_t(M).
\end{align*}
We can now take $p_0$ large depending only on $c_0$ and $C_5$ to ensure that for $p \geq p_0$ the inequality
\begin{align*}
\frac{d}{dt} \int_M  f_{\sigma,k}^{p} \,d\mu_t &\leq (C_3+k)^p \sigma \int_M |h|^2 \, d\mu_t  +C_4 \sigma^{-4/\eta} p \int_M f_{\sigma,k}^{p} \,d\mu_t \\
&+ C_4 \sigma^{-4/\eta} \mu_t(M).
\end{align*}
holds for almost every $t \in [0,T)$. 

Taking $k_0$ a bit larger depending on $n$ and $C_3$, using $k \geq k_0$ we can bound
\begin{align*}
\frac{d}{dt} \int_M  f_{\sigma,k}^{p} \,d\mu_t &\leq 2^p k^p \sigma \int_M |H|^2 \, d\mu_t  +C_4 \sigma^{-4/\eta} p \int_M f_{\sigma,k}^{p} \,d\mu_t \\
&+ C_4 \sigma^{-4/\eta} \mu_t(M).
\end{align*}
Since 
\[\frac{d}{dt} \int_M 2^p k^p \sigma \, d\mu_t = - 2^p k^p \sigma \int_M |H|^2 \,d\mu_t\]
this implies 
\begin{align*}
\frac{d}{dt} \int_M  f_{\sigma,k}^{p} + 2^p k^p \sigma \,d\mu_t &\leq  C_4 \sigma^{-4/\eta} p \int_M f_{\sigma,k}^{p} \,d\mu_t + C_4 \sigma^{-4/\eta} \mu_t(M)\\
&=  C_4 \sigma^{-4/\eta} p \int_M f_{\sigma,k}^{p} + p^{-1} \,d\mu_t .
\end{align*}
Hence the function 
\[\varphi(t):= \int_M f_{\sigma,k}^p +2^pk^p\sigma + p^{-1}  \,d\mu_t \]
satisfies 
\[\varphi'(t) \leq  C_4 \sigma^{-4/\eta} p \varphi(t)\]
for almost every $t \in [0,T)$. Since $\varphi$ is Lipschitz continuous in time it follows that 
\[\varphi(t) \leq \varphi(0) \exp( C_4 \sigma^{-4/\eta} p t).\]
In particular, $\varphi$ can be bounded from above in terms of its value at the initial time, and the constants $C_4$, $\eta$, $\sigma$, $p$ and $T$. Recall that $C_4$ depends only on $n$, $\eta$, $\Lambda$ and $L$. Also, 
\[f_{\sigma, k}^{p} \leq C_0^p |H|^{\sigma p},\]
so $\varphi(0)$ can be bounded purely in terms of $n$, $\Lambda$, $\sigma$, $p$, $L$ and $\mu_0(M)$. This completes the proof.
\end{proof}

\subsection{The supremum estimate}

Combining the $L^p$-estimates just established with the Michael-Simon Sobolev inequality \cite{Michael-Simon73} we obtain the following iteration inequality. The proof is very similar to that of Theorem 5.1 in \cite{Huisken84}, so we omit the details.

\begin{proposition}
There are positive constants $p_1 \geq p_0$ and $\ell_1 \leq \ell_0$ depending on $n$, $\varepsilon_0$, $\eta$, $\varepsilon$ and $\Lambda$, and a positive constant $k_1 \geq k_0$ depending on $n$,  $\varepsilon_0$, $\eta$, $\varepsilon$, $\Lambda$ and $L$, with the following property. Suppose $p \geq p_1$ and $\sigma \leq \ell_1 p^{-\frac{1}{2}}$ and set 
\[A(k) := \int_0^T \int_{\supp(f_{\sigma,k}(\cdot,t))} \,d\mu_t dt.\]
Then for every $h>k\geq k_1$ we have  
\begin{align*}
A(h) \leq \frac{C}{(h-k)^p} A(k)^{\gamma}.
\end{align*}
where $\gamma >1$ depends on $n$ and $C = C(n, \varepsilon_0, \eta, \varepsilon, \Lambda, L, \mu_0(M), T, \sigma, p)$. 
\end{proposition}

Appealing to Stampacchia's lemma (see for example Lemma B.1 in \cite{Kind-Stamp}) we obtain:

\begin{corollary}
There is a constant $k_2 = k_2 (n,\varepsilon_0, \eta, \varepsilon, \Lambda, L, \mu_0(M), T)$ such that 
\[f_{\sigma_0, k_2} \equiv 0\]
on $M\times [0,T)$, where $\sigma_0 := \ell_1 p_1^{-\frac{1}{2}}$ depends only on $n$, $\varepsilon_0$, $\eta$, $\varepsilon$ and $\Lambda$.  
\end{corollary}
Recall that $\eta$ depends only on $n$ and $\varepsilon_0$, and we chose $\Lambda$ depending only on $n$, $\varepsilon_0$ and $\varepsilon$. Therefore, by the corollary we have an estimate of the form
\begin{align*}
\frac{\lambda_1 + \varepsilon w}{|H|} \geq - C|H|^{-\sigma_0} -\Lambda \frac{|\hat A|^2}{|H|^2} 
\end{align*}
on $M\times [0,T)$, where $C = C(n,\varepsilon_0, \varepsilon, L, \mu_0(M), T)$. Appealing to the codimension estimate of Theorem \ref{thm:codim}, we finally obtain
\begin{equation}
\label{eq:conv_est}
\frac{\lambda_1 + \varepsilon w}{|H|} \geq - C|H|^{-\sigma_0} -C|H|^{-2\eta},
\end{equation}
where $C$ has the same dependencies as before. From here, since $\varepsilon$ can be made arbitrarily small, an application of Young's inequality to the two lower-order terms on the right-hand side gives the convexity estimate of Theorem \ref{thm:convex} (note that $T$ can be bounded in terms of $M_0$ by applying the maximum principle to the evolution equation of $W$).

\begin{remark}
There is another way to prove Theorem \ref{thm:convex} using compactness and the strong maximum principle, which we now sketch. Let $F:M \times [0,T) \to \mathbb{R}^{n+k}$ be a quadratically pinched mean curvature flow with $n \geq 5$ and $c < c_n$, fix a constant $\varepsilon >0$, and set 
\[f_\varepsilon := \frac{\lambda_1 + \varepsilon w}{|H|}.\]
For each $j \in \mathbb{N}$ set
\[\delta_{j} := \inf \{f_\varepsilon(x,t): |H|(x,t) \geq j\}.\]
Then the $\delta_j$ form a bounded nondecreasing sequence and therefore converge to some $\delta$. Choose a sequence $(x_j, t_j) \in P_j$ such that $f_\varepsilon(x_j,t_j) \to \delta$ and form a sequence of solutions by shifting $F(x_j,t_j)$ to the spacetime origin in $\mathbb{R}^{n+k} \times \mathbb{R}$ and parabolically rescaling by $|H|(x_j,t_j)$. Then the gradient estimates established in Section 3 of \cite{Nguyen2018a} ensure that this sequence converges smoothly in a small spacetime neighbourhood about the origin. The limit lies in $\mathbb{R}^{n+1}$ by Naff's codimension estimate, and the quantity $f_\varepsilon$ attains its minimum $\delta$ at the spacetime origin. The strong maximum principle applied to the evolution equation for $f_\varepsilon$ shows that $|\nabla h|\equiv 0$ on the limiting solution, which is consequently either a piece of shrinking sphere or cylinder. In either case we have $\lambda_1 \geq 0$, so $\delta >0$. In other words, on the original solution, there exists a threshold $C_\varepsilon$ depending on $\varepsilon$ and $M_0$ such that whenever $|H|(x,t) \geq C_\varepsilon$ there holds 
\[ \frac{\lambda_1(x,t)}{|H|(x,t)} \geq -\varepsilon \frac{w(x,t)}{|H|(x,t)} \geq - C(n) \varepsilon.\]

Note that this argument does not yield a quantitative blow-up rate for the negative part of the second fundamental form, in contrast to the estimate \eqref{eq:conv_est} proven above. Moreover, the gradient estimates in \cite{Nguyen2018a} are difficult to establish and will not be available in other situations where the Stampacchia iteration goes through. Indeed, an estimate showing asymptotic positivity of curvature is often needed to establish gradient estimates; this is the case for the fully nonlinear flow studied in \cite{Bren-Huisk17} and for three-dimensional Ricci flow \cite{Perelman_a}.
\end{remark}

\section{Singularity formation}

\label{sec:singularity_formation}

In the study of parabolic evolution equations it is natural to distinguish between singularities which form at different rates. For a solution of mean curvature flow $F:M\times[0,T) \to \mathbb{R}^{n+k}$ where $T$ is the maximal time we say that a type I singularity forms as $t \to T$ if there is a positive constant $C$ such that 
\[\max_{M_t} |A|^2 \leq \frac{C}{T-t}.\]
Note that this is the blow-up rate for solutions which shrink homothetically (such as shrinking spheres and cylinders). If on the other hand 
\[\limsup_{t \nearrow T} \Big[(T-t)\max_{M_t} |A|^2 \Big] =\infty\]
then the singularity forming at time $T$ is said to be of type II. 

For certain type I singularities, Baker used Huisken's monotonicity formula to show that appropriate rescalings about the singularity converge to a homothetically shrinking solution \cite{Baker_thesis}. Moreover, Baker could show that the only such solutions satisfying the quadratic pinching condition are shrinking spheres and (generalised) cylinders. The analogous result for mean-convex solutions of codimension one was proven earlier by Huisken \cite{Huisken90}. In \cite{Huisk-Sin99a} Huisken and Sinestrari used their convexity estimate to show that at a type II singularity, appropriate rescalings about the maximum of the curvature converge to a convex translating solution. In this section we use our convexity estimate to generalise their result to higher codimensions.

Fix a smooth mean curvature flow $F:M\times[0,T) \to \mathbb{R}^{n+k}$ of dimension $n \geq 5$ which is quadratically pinched with $c < c_n$, and suppose a type II singularity is forming as $t \to T$. Consider a sequence of times $\tilde t_j \to T$ and let $(x_j,t_j)$ be such that 
\[(\tilde t_j - t_j) |H|^2(x_j,t_j) := \max_{M \times [0,\tilde t_j]} (\tilde t_j - t) |H|^2(x,t).\]
Then we have 
\[|H|^2(x_j, t_j) = \max_{M} |H|^2(x,t_j) \]
By the type II assumption, for each $K > 0$ there is a point $(y,\tau) \in M\times [0,T)$ such that 
\[(T-\tau) |H|^2(y,\tau) \geq K.\]
If $j$ is large enough so that $\tilde t_j > \tau$ then we have 
\begin{align*}
(\tilde t_j - t_j) |H|^2(x_j,t_j) = (\tilde t_j - \tau)|H|^2(y,\tau) \geq K - (T - \tilde t_j) |H|^2(y,\tau).
\end{align*}
Hence if $j$ is sufficiently large there holds 
\[(\tilde t_j - t_j)|H|^2(x_j, t_j) \geq K/2,\]
and since $K$ can be made arbitrarily large this shows that 
\[(\tilde t_j - t_j)|H|^2(x_j, t_j) \to \infty.\]
It follows that $t_j \to T$. 

Let $L_j^2 := |H|^2(x_j,t_j)$ and consider the sequence of rescaled solutions defined by 
\[F_j(x,t) := L_j (F(x,L_j^{-2}t + t_j ) - F(x_j,t_j)), \qquad (x,t) \in M\times [- L_j^2t_j, L_j^2(T- t_j)),\]
which satisfy the conditions
\[F_j(0,0) = 0, \qquad |H_j|^2(0,0) = 1,\]
where $H_j$ is the mean curvature vector of $F_j$. More generally, for $t \leq L_j^2 (\tilde t_j - t_j)$ there holds
\begin{align*}
|H_j|^2(x,t) &= L_j^{-2} |H|^2(x, L_j^{-2} t + t_j) \\
&=L_j^{-2} \frac{(\tilde t_j - L_j^{-2} t -t_j ) |H|^2(x, L_j^{-2} t + t_j)}{\tilde t_j - L_j^{-2} t - t_j}  \\
& \leq L_j^{-2} \frac{(\tilde t_j - t_j) |H|^2(x_j, t_j) }{\tilde t_j - L_j^{-2} t - t_j}\\
& = \frac{\tilde t_j - t_j}{\tilde t_j - t_j -  L_j^{-2} t }.
\end{align*}
Therefore, for times $t \leq \delta L_j^2 (\tilde t_j - t_j)$ with $\delta <1$ we have
\[\max_{M} |H_j|^2(\cdot,t) \leq \frac{1}{1-\delta}.\]
Passing to a subsequence in $j$, we can guarantee that there is a sequence $\tau_j \to \infty$ such that 
\[\max_{M} |H_j|^2(\cdot, t) \leq 1 + \frac{1}{j}, \qquad \forall \; t \in [-\tau_j, \tau_j]. \]

It is well known that for a compact solution of mean curvature flow, a global upper bound for $|A|$ implies bounds on all of the higher derivatives of $A$. This follows from the estimates in \cite{Ecker-Huisken} in the codimension-one case, and similar arguments work in higher codimensions (the details can be found in Section 4.3 of \cite{Baker_thesis}). Standard compactness theorems therefore imply that there is a smooth solution 
\[\tilde F : \tilde M \times (-\infty, \infty ) \to \mathbb{R}^{n+k}\]
such that the sequence $F_j$ subconverges smoothly to $\tilde F$ in the following local sense. There is a sequence of nested open sets $U_l \subset \tilde M$ such that 
\[\tilde M = \bigcup_{l \in \mathbb N} U_l\]
and local diffeomorphisms $\varphi_l : U_l \to M$ such that the sequence 
\[(x,t) \mapsto F_j(\varphi_l(x), t), \qquad (x,t) \in U_l \times [-l, l] \]
converges smoothly to 
\[(x,t) \mapsto \tilde F(x,t), \qquad (x,t) \in U_l \times [-l, l]\]
as $j \to \infty$ for every $l \in \mathbb{N}$. This follows for example from Hamilton's compactness theorem \cite{Ham_compactness}, as is illustrated in Section 6.1 of \cite{Baker_thesis}.

\begin{theorem}
The smooth limiting solution $\tilde F : \tilde M \times (-\infty, \infty) \to \mathbb{R}^{n+k}$ obtained by the above rescaling procedure lies in an $(n+1)$-dimensional affine subspace and is either: a strictly convex translating solution; or the isometric product of $\mathbb{R}^m$ with a strictly convex translating solution of dimension $n -m$. 
\end{theorem}
\begin{proof}
We denote the second fundamental form of $\tilde F$ by $\tilde A$, and use this convention for other curvature quantities as well. We know that the mean curvature vector of $\tilde F$ satisfies $|\tilde H|(0,0) = 1$, but a priori, there could be a point on the limiting solution where $|\tilde H| = 0$. However the  evolution equation 
\[(\partial_t - \Delta ) \tilde W \geq 2 |\tilde h|^2 \tilde W \geq \frac{3}{2} \tilde W^3\]
is still valid, so by the strong maximum principle we either have $\tilde W >0$ or $\tilde W \equiv 0$. In the latter situation $|\tilde H|^2(0,0) = 0$, which is a contradiction, so appealing to the pinching we conclude that
\[\frac{4}{3n} |\tilde H|^2  > \tilde W >0 \]
on $\tilde M \times (-\infty, \infty)$. Hence Naff's codimension estimate implies $\tilde h \equiv \tilde A$, and consequently, the image of $\tilde F$ lies in an $(n+1)$-dimensional subspace of $\mathbb{R}^{n+1}$. By the convexity estimate, $\tilde h \geq 0$ on $\tilde M \times (-\infty,\infty)$. 

To recap, the blow-up limit $\tilde F$ is codimension one, has nonnegative second fundamental form, its scalar mean curvature $|\tilde H|$ is globally bounded from above by one, and this global upper bound is attained at the spacetime origin. This is exactly the situation considered in Section 4 of \cite{Huisk-Sin99a}. Applying Hamilton's strong maximum principle for tensors to the evolution of the second fundamental form,
\[(\partial_t - \Delta) \tilde h_j^i = |\tilde h|^2 \tilde h^i_j,\]
we conclude that the solution $\tilde M_t := \tilde  F(\tilde M, t)$ splits as an isometric product $\mathbb{R}^{m} \times N_t$, where $N_t$ is a strictly convex solution of dimension $n -m$ which exists for all $ t \in (-\infty, \infty)$. Since the spacetime maximum of the mean curvature of $N_t$ is attained at the spacetime origin, the rigidity case of Hamilton's Harnack inequality \cite{Ham_harnack} implies the family $N_t$ moves by translation. 
\end{proof}

\begin{remark}
By the gradient estimate in \cite{Nguyen2018a}, the limiting flow $\tilde F$ cannot be the product of $\mathbb{R}^{n-1}$ with a grim reaper. On the other hand, the grim reaper is the only strictly convex translator in $\mathbb{R}^2$, so we conclude that $m \leq n-2$. 
\end{remark}

\begin{remark}
If $N_t$ is uniformly two-convex in the sense that the smallest two principal curvatures satisfy 
\[\lambda_1 + \lambda_2 \geq \alpha H\]
globally for some $\alpha >0$, then by the gradient estimate in \cite{Nguyen2018a} and work of Bourni-Langford \cite{Bourni-Langford}, $N_t$ must be rotationally symmetric and hence a bowl soliton of dimension $m$. See also the paper \cite{Naff2019a}.
\end{remark}


\begin{thebibliography}{Ham95b}

\bibitem[AB10]{Andrews2010}
Ben Andrews and Charles Baker.
\newblock Mean curvature flow of pinched submanifolds to spheres.
\newblock {\em J. Differential Geom.}, 85(3):357--395, 2010.

\bibitem[Bak11]{Baker_thesis}
Charles Baker.
\newblock The mean curvature flow of submanifolds of high codimension.
\newblock {\em arXiv preprint arXiv:1104.4409}, 2011.

\bibitem[BH17]{Bren-Huisk17}
Simon Brendle and Gerhard Huisken.
\newblock A fully nonlinear flow for two-convex hypersurfaces in {R}iemannian
  manifolds.
\newblock {\em Invent. Math.}, 210(2):559--613, 2017.

\bibitem[BL16]{Bourni-Langford}
Theodora Bourni and Mat Langford.
\newblock Type-{II} singularities of two-convex immersed mean curvature flow.
\newblock {\em Geometric Flows}, 2(1), 2016.

\bibitem[Bre15]{Brendle15}
Simon Brendle.
\newblock A sharp bound for the inscribed radius under mean curvature flow.
\newblock {\em Invent. Math.}, 202(1):217--237, 2015.

\bibitem[EG15]{Ev-Gar}
Lawrence~Craig Evans and Ronald~F Gariepy.
\newblock {\em Measure theory and fine properties of functions}.
\newblock CRC press, 2015.

\bibitem[EH91]{Ecker-Huisken}
Klaus Ecker and Gerhard Huisken.
\newblock Interior estimates for hypersurfaces moving by mean curvature.
\newblock {\em Invent. Math.}, 105(3):547--569, 1991.

\bibitem[Ham95a]{Ham_compactness}
Richard~S Hamilton.
\newblock A compactness property for solutions of the {R}icci flow.
\newblock {\em American journal of mathematics}, 117(3):545--572, 1995.

\bibitem[Ham95b]{Ham_harnack}
Richard~S Hamilton.
\newblock Harnack estimate for the mean curvature flow.
\newblock {\em Journal of Differential Geometry}, 41(1):215--226, 1995.

\bibitem[HS99]{Huisk-Sin99a}
Gerhard Huisken and Carlo Sinestrari.
\newblock Convexity estimates for mean curvature flow and singularities of mean
  convex surfaces.
\newblock {\em Acta mathematica}, 183(1):45--70, 1999.

\bibitem[HS09]{Huisk-Sin09}
Gerhard Huisken and Carlo Sinestrari.
\newblock Mean curvature flow with surgeries of two-convex hypersurfaces.
\newblock {\em Invent. Math.}, 175(1):137--221, 2009.

\bibitem[Hui84]{Huisken84}
Gerhard Huisken.
\newblock Flow by mean curvature of convex surfaces into spheres.
\newblock {\em J. Differential Geom.}, 20(1):237--266, 1984.

\bibitem[Hui90]{Huisken90}
Gerhard Huisken.
\newblock Asymptotic behavior for singularities of the mean curvature flow.
\newblock {\em Journal of Differential Geometry}, 31(1):285--299, 1990.

\bibitem[KS80]{Kind-Stamp}
David Kinderlehrer and Guido Stampacchia.
\newblock {\em An introduction to variational inequalities and their
  applications}, volume~31.
\newblock Siam, 1980.

\bibitem[Lan17]{Langford17}
Mat Langford.
\newblock A general pinching principle for mean curvature flow and
  applications.
\newblock {\em Calc. Var. Partial Differential Equations}, 56(4):Paper No. 107,
  31, 2017.

\bibitem[MS73]{Michael-Simon73}
J.~H. Michael and L.~M. Simon.
\newblock Sobolev and mean-value inequalities on generalized submanifolds of
  $\mathbb{R}^n$.
\newblock {\em Communications on Pure and Applied Mathematics}, 26(3):361--379,
  1973.

\bibitem[Naf19a]{Naff2019}
Keaton Naff.
\newblock Codimension estimates in mean curvature flow.
\newblock arXiv:1906.08184 [math.DG], 2019.

\bibitem[Naf19b]{Naff2019a}
Keaton Naff.
\newblock Singularity models of pinched solutions of mean curvature flow in
  higher codimension.
\newblock arXiv:1910.03968 [math.DG], 2019.

\bibitem[Ngu18]{Nguyen2018a}
Huy~The Nguyen.
\newblock Cylindrical estimate for high codimension mean curvature flow.
\newblock arXiv:1805.11808 [math.DG], 2018.

\bibitem[Ngu20]{Nguyen20}
Huy~The Nguyen.
\newblock High codimension mean curvature flow with surgery.
\newblock arXiv:2004.07163 [math.DG], 2020.

\bibitem[Per02]{Perelman_a}
Grisha Perelman.
\newblock The entropy formula for the ricci flow and its geometric
  applications.
\newblock {\em arXiv preprint math/0211159}, 2002.

\bibitem[Whi03]{White}
Brian White.
\newblock The nature of singularities in mean curvature flow of mean-convex
  sets.
\newblock {\em J. Amer. Math. Soc.}, 16(1):123--138, 2003.

\end{thebibliography}
\end{document}